\documentclass[12pt]{amsart}

\usepackage[utf8]{inputenc}
\usepackage{graphicx}
\usepackage{graphicx,color}
\usepackage{latexsym}
\usepackage[all]{xy}
\usepackage{mathrsfs}
\usepackage{enumerate}
\usepackage[shortlabels]{enumitem}
\usepackage{multicol}
\usepackage{lipsum}
\usepackage{amssymb}
\usepackage{tikz}
\usetikzlibrary{cd}
\usetikzlibrary{patterns}
\definecolor{DimGray}{rgb}{0.41, 0.41, 0.41}
\usepackage{float}
\usepackage{bm}
\usepackage{mathptmx}
\usetikzlibrary{shapes.geometric}
\usepackage{dsfont}
\usepackage{cjhebrew}
\usepackage[normalem]{ulem}
\usepackage[alpine]{ifsym}
\usepackage{xpicture}
\usepackage{calculator}
\usepackage{graphicx}
\usepackage{pgf,tikz,pgfplots}
\pgfplotsset{compat=1.15}
\usetikzlibrary{arrows}

\usepackage{amsmath}
\usepackage{amsfonts}
\usepackage{mathrsfs}
\usepackage{amssymb,enumerate}
\usepackage{amsthm}
\usepackage{fancyhdr}
\usepackage{hyperref}
\usepackage{cleveref}

\usepackage{caption}
\usepackage{subcaption}
\usepackage{soul} 

\usepackage[backend=biber,style=numeric, maxbibnames=99,doi=false,isbn=false,url=false]{biblatex}
\renewbibmacro{in:}{%
 \ifentrytype{article}{}{\printtext{\bibstring{in}\intitlepunct}}}

\def\sideremark#1{\ifvmode\leavevmode\fi\vadjust{\vbox to0pt{\vss 
			\hbox to 0pt{\hskip\hsize\hskip1em           
				\vbox{\hsize2cm\tiny\raggedright\pretolerance10000
					\noindent #1\hfill}\hss}\vbox to8pt{\vfil}\vss}}} %
\theoremstyle{plain}
\newtheorem{theorem}{Theorem}[section]
\newtheorem{proposition}[theorem]{Proposition}

\newtheorem*{theorem*}{Theorem}
\newtheorem{lemma}[theorem]{Lemma}
\newtheorem{corollary}[theorem]{Corollary}
\newtheorem{conj}[theorem]{Conjecture}

\theoremstyle{definition}
\newtheorem{defin}[theorem]{Definition}

\newtheorem{ex}[theorem]{Example}

\theoremstyle{remark}
\newtheorem{rem}[theorem]{Remark}
\newcommand{\ord}{\operatorname{ord}}

\newcommand{\frakg}{\mathfrak{g}}

\newcommand{\vu}{\underline{v}}

\newcommand{\con}{\mathbf{c}}
\newcommand{\obeta}{\overline{\beta}}

\newcommand{\pr}{\mathrm{pr}}
\newcommand{\ind}{\mathtt{I}}
\newcommand{\jnd}{\mathtt{J}}

\counterwithin{figure}{section}
\numberwithin{equation}{section}

\title[Tjurina number for a plane curve with two branches]{The Tjurina number of a plane curve with two branches and high intersection multiplicity}



\author{Patricio Almir\'on}

\author{Marcelo E. Hernandes}

\thanks{The first named author is supported by Grant RYC2021-034300-I funded by MICIU/AEI/10.13039/501100011033 and by European Union NextGenerationEU/PRTR and during the elaboration of this work was also supported by the IMAG–Maria de Maeztu grant CEX2020-001105-M / AEI /10.13039/501100011033 and by Spanish Ministerio de Ciencia, Innovaci\'{o}n y Universidades PID2020-114750GB-C32/AEI/10.13039/501100011033. The second named author is partially supported by CNPq-Brazil.}

\address{Departamento de Matemáticas, Universidad de Zaragoza\\
C. Pedro Cerbuna 12 \\
50009 Zaragoza, Spain.}
\email{patricioalmiron@ugr.es; palmiron@unizar.es}

\address{Departamento de Matem\'atica, Universidade Estadual de Maring\'a\\ Av. Colombo 5790 \\ 87020900 Maring\'a, Paran\'a, Brazil.}
\email{mehernandes@uem.br}

\begin{document}

\begin{abstract}
In this paper we provide a family of reduced plane curves with two branches that have a constant Tjurina number in their equisingularity class, along with a closed formula for it in terms of topological data. 
\end{abstract}

\maketitle


\section{Introduction}
Let \((C_f,\boldsymbol{0})\) be a germ of a reduced plane curve singularity with \(r \geq 1\) branches. One of its main analytic invariants is the Tjurina number, which can be computed as 
\[\tau(C_f)=\dim_{\mathbb{C}}\frac{\mathbb{C}\{x,y\}}{(f,f_x,f_y)},\]
where \(f=0\) is a defining equation for \((C_f,\boldsymbol{0}).\) Despite being a well-studied invariant, several important problems remain open. Among these, determining a closed formula for the minimal Tjurina number in a fixed topological class in terms of topological invariants stands out as one of the most challenging. A closely related and even more difficult problem is the characterization of plane curves with a constant Tjurina number within an equisingularity class.
\medskip

While determining a closed formula for the minimal Tjurina number in a fixed equisingularity class is completely solved for \(r=1\) (see \cite{Altaumin,GHtaumin}), it remains open for \(r > 1\). In fact, for \(r > 1\) a recursive formula was proven by Brian\c{c}on, Granger, and Maisonobe \cite{BGM88} for positive weight deformations of a semi-quasihomogeneous curve \(x^a + y^b\), i.e. plane curves with \(r = \gcd(a, b)\) branches that are all topologically equivalent and have a single Puiseux pair. Recently, Genzmer \cite{Genzmer22,Genzmer24} has provided an algorithm that allows the computation of the minimal Tjurina number in a fixed topological class. However, a closed formula, or even a manageable recursive one in terms of topological data, is still far beyond reach at this moment.
\medskip

The search for curves with a constant Tjurina number in the equisingularity class is even less developed. Except for cases where the equisingularity class contains only one analytic representative,
there are few known irreducible cases: 
irreducible plane curve singularities with value semigroup \(\langle \obeta_0,\dots,\obeta_g\rangle\) satisfying \(\gcd(\obeta_0,\dots,\obeta_{g-1})=2\). The case \(g=2\) was first described in 1990 by Luengo and Pfister \cite{LuengoPfister88}, and the general case for \(g \geq 2\) was provided by Abreu and the second-named author \cite{hernandessemiroots} in 2022. Up to best of the authors' knowledge, there are no cases in the literature of plane curves with $r>1$ branches in which all different analytical representatives in a class of equisingularity have the same Tjurina number.

\medskip

The main goal of this paper is to present a family of plane curve singularities with two branches that have a constant Tjurina number in their equisingularity class but whose branches can achieve distinct Tjurina numbers, along with a closed formula for this number in terms of topological data. 
\medskip

To provide such a family, we have used the ideas of \cite{hernandessemiroots}. There, a computation of the Tjurina number in terms of the set of values of K\"{a}hler differentials plays a crucial role. For an irreducible plane curve, Berger \cite{Berger63} shows that $\mu(C)-\tau(C)=\sharp(\Lambda_f\setminus S)$
where \(\mu(C)\) is the Milnor number, $S$ is the semigroup of values of \(C\) and $\Lambda_C$ is the set of values of the K\"{a}hler differentials of $C$. Our first main result, Theorem \ref{thm:Tjurinaformula}, presents a generalization of this result for any reduced plane curve. More precisely, we show that for a reduced plane curve \(C\) with \(r\geq 1\) branches we always have
\[\mu(C)-\tau(C)=d(\overline{\Lambda_C}, S)\]
where $d(\overline{\Lambda_f},S)$ indicates the distance between the values set $\overline{\Lambda_f}:=\Lambda_f\cup\{0\}$ and $S$ as presented in Section 2. This result has its own interest as it is provided for any \(r\geq 1\) and presents an alternative method to compute the Tjurina number of a reduced plane curve to the ones presented in \cite{bayeretal,hefezhernandes24}.
\medskip

The next step is a careful study of the set of values of the K\"{a}hler differentials  $\Lambda_C$ . In  \cite{bayeretal,hefezhernandes24} some relations are presented that allow us to express $\tau(C)$ in terms of the Tjurina number of irreducible components $\{C_i;\ i=1,\ldots ,r\}$ of $C$, the intersection multiplicity of $[C_i,C_j]_0$ and an analytical invariant $\Theta_i$ for $2\leq i\leq r$ related with the values of K\"ahler differentials of $C$. In Section 3, we study the set $\Lambda_C$ for $C=C_1\cup C_2$ considering $C_1$ and $C_2$ equisingular plane branches with value semigroup $\langle \overline{\beta}_0,\ldots ,\overline{\beta}_g\rangle$ and $I:=[C_1,C_2]_0>e_{g-1}\overline{\beta_g}$ where $e_{g-1}=gcd(\overline{\beta}_0,\ldots ,\obeta_{g-1})$. In this case, considering logarithmic differentials, we describe the infinite fibers of $\Lambda_C$ (see Subsection 3.1) that allow us to compute $\Theta_2$ and to show that the conductor of $\Lambda_C$ depends only on $[C_1,C_2]_0$ and the multiplicity of $C_i$ (see Theorem \ref{thm:conductorkahler1}). Therefore, we show that in this case the conductor only depends on topological data and thus it is independent of the analytic type of the curve. This result also has its own interest as for the family of irreducible plane curve singularities with value semigroup \(\langle \obeta_0,\dots,\obeta_g\rangle\), where \(\gcd(\obeta_0,\dots,\obeta_{g-1})=2\) provided in \cite{hernandessemiroots} the conductor of \(\Lambda\) may change but the Tjurina number is still constant. Then, our family with two branches has even a stronger behavior in this sense.
\medskip

These two ingredients, the generalization of Berger's result and the description of the infinite fibers of \(\Lambda_C,\) are enough to achieve our main goal. In Section \ref{sec:Tjurina2}, considering $C=C_1\cup C_2$ such that $C_1$ and $C_2$ sharing the same value semigroup $\langle \overline{\beta}_0,\ldots ,\overline{\beta}_g\rangle$ and $I>e_{g-1}\overline{\beta_g}$, we prove (Theorem \ref{thm:taumin1})
\[\tau(C)=2I+\mu(C_1)-1\]
that is, the Tjurina number is given by only topological datas, i.e. the intersection multiplicity and the Milnor number \(\mu(C_1)=\mu(C_2)\) of the branches. This result gives us new examples of topological classes where the Tjurina number is constant. Moreover, in contrast to the irreducible case, the number of possible plane curves with two branches belonging to this family is quite large. In addition, our formula allows to provide a new proof in this particular case of the inequality \(\mu(C)/\tau(C)< 4/3\) provided by the first named author \cite{quotpat} for any plane curve.
\medskip

A deeper analysis of our results reveals that achieving similar results in more general cases presents an interesting yet challenging problem. To conclude the paper, we discuss some of the key difficulties in these cases and propose conjectures that could guide future investigations.

\section{Value sets of fractional ideals}\label{sec:semigroup}
Let \(C_f\) be a (reduced) germ of complex plane curve singularity with equation $f=\prod_{i=1}^{r}f_i=0$ where $f_i\in\mathbb{C}\{x,y\}$ is irreducible. Each $f_i$ defines a branch $C_i$ and its analytical type is characterized by the local ring \(\mathcal{O}_i:=\mathbb{C}\{x,y\}/(f_i)\) up to $\mathbb{C}$-algebra isomorphism. The field of fraction $\mathcal{K}_i$  of $\mathcal{O}_i$ is isomorphic to $\mathbb{C}(t_i)$ and associated to it we have a canonical discrete valuation $v_i:\mathcal{K}_i\rightarrow\overline{\mathbb{Z}}:=\mathbb{Z}\cup\{\infty\}$. The isomorphic image of $\mathcal{O}_i$ in $\mathbb{C}(t_i)\cong \mathcal{K}_i$ can be given by $\mathbb{C}\{t_i^n,\sum_{j\geq n}a_jt_i^j\}$ where $n$ is the multiplicity of $f_i$, the set $\{n, j;\ a_j\neq 0\}$ does not admit a nontrivial common divisor and
\begin{equation}\label{puiseux-param}
\varphi_i(t_i)=\left (t_i^n,\sum_{j\geq n}a_jt_i^j\right )
\end{equation}
is a Newton-Puiseux parametrization of $C_i$. Notice that we have $f_i\left (t_i^n,\sum_{j\geq n}a_jt_i^j\right )=0$, or equivalently $f_i\left (x,\sum_{j\geq n}a_jx^\frac{j}{n}\right )=0$. We call $s_i(x)=\sum_{j\geq n}a_jx^\frac{j}{n}$ a Puiseux series of the branch $C_i$.

If $\overline{h}\in\mathcal{O}_i$ with $h\in\mathbb{C}\{x,y\}$, then we get $v_i(\overline{h})=\ord_{t_i}\varphi_i^*(h)=[f_i,h]_0$, where $[f_i,h]_0:=[C_{f_i},C_h]_0$ denotes the intersection multiplicity of $C_{f_i}$ and $C_h$ at the origin.

The image of $v_i$ of $\mathcal{O}_i\setminus\{0\}$, that is
\[S(C_i):=\{v_i(\overline{h})\in\mathbb{N}\ :\ \overline{h}\in\mathcal{O}_i, \overline{h}\neq 0\}\]
is the semigroup of values of $C_i$ and it is a numerical semigroup whose minimal generating set  can be computed from a Newton-Puiseux parametrization (cf. \cite{zariski}) as follows. 

If $f\in\mathbb{C}\{x,y\}$ is irreducible with Newton-Puiseux parametrization $(t^n,\sum_{j\geq n}a_jt^j)$, let $(z)\subseteq \mathbb{Z}$ be the set of multiples of $z$, $\beta_0=n$ and set
\begin{align*}
e_{i-1} &=\mathrm{gcd}\{ \beta_0, \ldots , \beta_{i-1}\}\\
\beta_i &= \min \{j : a_j \neq 0, j \notin (e_{i-1})\} \ \ \mbox{for}\;i=1,\dots, g\;\mbox{where}\;e_g=1.
\end{align*}
    
Let us define 
\begin{equation}\label{eq:definbetabarra}
\obeta_0=\beta_0,\;\obeta_1=\beta_1\quad\text{and}\quad \obeta_{i+1}=n_i\obeta_i+\beta_{i+1}-\beta_i\quad \mbox{where}\ n_i=e_{i-1}/e_i\ \mbox{for}\ 1\leq i< g.
\end{equation}

It follows (cf. \cite{zariski}) that the values semigroup \(S(C_f)\) is minimally generated by the elements \(\obeta_0,\dots,\obeta_g,\) i.e. 
\[
S(C_f)=\langle\obeta_0,\dots,\obeta_g\rangle=\big \{\gamma\in\mathbb{N}\;: \;\gamma=m_0\obeta_0+\cdots+m_g\obeta_g\ \ \mbox{with} \ \ m_i\in\mathbb{N},\ \ \mbox{for}\; \ i=0,\dots,g\big \}.
\]
Since $S(C_f)$ is a numerical semigroup it admits a conductor $\con(S(C_f))$, that is the minimun element in $S(C_f)$ such that $\con(S(C_f))+\mathbb{N}\subset S(C_f)$  and it can be computed (cf. \cite{zariski}) by \begin{equation}\label{fy-cond}
\con(S(C_f))=\sum_{i=1}^{g}(n_i-1)\obeta_i-\obeta_0+1=v(f_y)-\obeta_0+1.
\end{equation}

The topological (equisingularity) class of the branch $C_f$ is totally determined by its values semigroup. 

\medskip
For $f=\prod_{i=1}^{r}f_i$ with $r>1$ the topological class of $C_f$ is also characterized by a semigroup as we describe in the sequel. 

The total ring of fraction $\mathcal{K}$ of the local ring $\mathcal{O}=\mathbb{C}\{x,y\}/(f)$ is isomorphic to $\prod_{i=1}^{r}\mathbb{C}(t_i)$. If $\pi_i: \mathcal{K}\rightarrow \mathcal{K}_i$ denotes the natural projection then we can consider $\vu :\mathcal{K}\rightarrow \overline{\mathbb{Z}}^r$ defined by $\vu (q)=(v_1(q),\dots,v_r(q))$ where $q\in\mathcal{K}$ and $v_i(q)$ stands for $v_i(\pi_i(q))$. 

In what follows we set $\ind :=\{1,\ldots ,r\}$ and we consider the product order on $\mathbb{Z}^r$, that is, given $\alpha=(\alpha_1,\dots,\alpha_r), \beta=(\beta_1,\dots,\beta_r)\in \mathbb{Z}^r$,
\[
\alpha\leq \beta\Longleftrightarrow\;\alpha_i\leq\beta_i\;\,\mbox{for all}\,\ i \in \ind.
\]

The values semigroup of \(C_f\) is the additive submonoid of \(\mathbb{N}^r\) defined by \[ S=S(C_f):=\{\vu(h)=(v_1(h),\ldots,v_r(h))\in\mathbb{N}^r\; :\;h\in\mathcal{O},\,h\ \mbox{is not a divisor of}\ 0\}. \]

For $\jnd=\{j_1,\ldots ,j_k\}\subset \ind$ and $\alpha=(\alpha_1,\ldots ,\alpha_r)\in\mathbb{Z}^r$ we put $pr_J(\alpha)=(\alpha_{j_1},\ldots ,\alpha_{j_k})\in\mathbb{N}^k$. Notice that \(S\subsetneq S(C_1)\times\cdots\times S(C_r)\) and $pr_J(S)=S(C_{f_{\jnd}})$ where $f_J=\prod_{j\in\jnd}f_j$. 

Some elementary properties of the semigroup of values \(S\) are the following (see \cite{Delgmanuscripta1}):
\begin{enumerate}
	\item If \(\alpha,\beta\in S\), then 
 \[
 \min\{\alpha,\beta\}:=(\min\{\alpha_i,\beta_i\})_{i\in \ind}\in S.\]
	\item If \(\alpha,\beta\in S\) and \(j\in \ind\) with \(\alpha_j=\beta_j\), then there exists \(\epsilon\in S\) such that \(\epsilon_j>\alpha_j=\beta_j\) and \(\epsilon_i\geq\min\{\alpha_i,\beta_i\}\) for all \(i\in \ind\setminus\{j\}\), with equality if \(\alpha_i\neq\beta_i\). 
	\item The semigroup $S$ has a conductor \(\con_S=\con(S),\) which is defined to be the minimal element of \(S\) such that $\gamma\in S$ whenever \(\gamma\geq \con_S\). Moreover, we get (see \cite{Delgmanuscripta1}) 
 \begin{equation}\label{conductor}
 \con_S=(c_1+\sum_{i\in\ind\setminus\{1\}}I_{1,i},\ldots ,c_r+\sum_{i\in\ind\setminus\{r\}} I_{r,i}),
 \end{equation} where $c_i=\con(S(C_i))$ and $I_{i,j}=[f_i,f_j]_0$.
\end{enumerate}

\begin{rem}\label{cond-milnor} The conductor of the semigroup of a plane curve $C$ defined by $f$ is closely related to its Milnor number $\mu(C)=\dim_{\mathbb{C}}\mathbb{C}\{x,y\}/(f_x,f_y)$. In fact, if $f$ is irreducible then $\mu(C)=\con_S$ and for $f=\prod_{i=1}^rf_i$ we get
$$\mu(C)=\sum_{i=1}^rpr_i(\con_S)-r+1=\sum_{i=1}^rc_i+2\sum_{1\leq i<j\leq r}I_{i,j}-r+1.$$
\end{rem}

For \(r>1,\) the semigroup $S$ is no longer finitely generated, but it is finitely determined (see \cite{Delgmanuscripta1,CDGlondon}). However, if we allow $\vu(h)\in S\subset \overline{\mathbb{Z}}^r$ for {\it any} $h\in\mathcal{O}$, that is, $v_i(0)=\infty$, then $(S,\min ,+)$ is a finite generated semiring (see \cite{hernandesSemiring}).

Now, for a given \(A\subseteq \overline{\mathbb{Z}}^r,\) \(\alpha\in\mathbb{Z}^r\) and an index subset $\emptyset\neq \jnd\subset \ind$, we define
\[
F_\jnd(A,\alpha)=\big\{\beta\in A\;:\;\beta_j=\alpha_j\quad\forall j\in \jnd\quad \text{and}\quad \beta_k>\alpha_k\quad\forall k\notin \jnd \big\},
\]
\[
\overline{F}_\jnd(A,\alpha)=\big\{\beta\in A\;:\;\beta_j=\alpha_j\quad\forall j\in \jnd\quad \text{and}\quad \beta_k\geq\alpha_k\quad\forall k\notin \jnd \big\}.
\]
The fiber of \(\alpha\) in \(A\) is defined as \(F(A,\alpha)=\cup_{i=1}^{r}F_i(A,\alpha).\)

The fibers \(F(S,\alpha)\) are important in order to determine \(S\) in terms of its projections (see \cite{Delgmanuscripta1}). 

An element \(\gamma\in S\) is called a maximal element of \(S\) if \(F(S,\gamma)=\emptyset.\) If, moreover,  \(F_\jnd(S,\gamma)=\emptyset\) for all \(\jnd\subset \ind\) such that \(\emptyset\neq \jnd\neq \ind\), then \(\gamma\) is said to be absolute maximal. On the other hand, if \(\gamma\) is a maximal and if \(F_\jnd(S,\alpha)\neq \emptyset\) for all \(\jnd\subset \ind\) such that \(\sharp\jnd\geq 2\), then \(\gamma\) will be called relative maximal. It is easily checked that the set of maximal elements of \(S\) is finite.

\begin{defin}\label{inf-fiber} Let $A\subset\overline{\mathbb{Z}}^r$ be a set satisfying the properties (1), (2) and (3). We will say that \(\alpha\) has infinite fiber in \(A\) with respect to \(\jnd\subset \ind\), writing $F_{\jnd}^A(\alpha)=\infty$, if there exists \(\beta\in \overline{F}_\jnd(A,\alpha)\) such that \(pr_{\ind\setminus\jnd}(\beta)\geq pr_{\ind\setminus\jnd}(\con_A).\)
\end{defin}

Observe that \(\alpha\) has infinite fiber in \(A\) with respect to \(\jnd\subset \ind\) is the same as saying that
\[\{ \delta \in \mathbb{N}^r \mid pr_\jnd(\delta) = pr_\jnd(\alpha), pr_{\ind\setminus\jnd}(\delta) \geq pr_{\ind\setminus\jnd}(\con_A) \} \subseteq A,\]
or equivalently, there exists $\beta\in F_{\jnd}(A,\alpha)$ such that $\beta+\overline{F}_{\jnd}(\mathbb{N}^r,0)\subset \overline{F}_{\jnd}(A,\alpha)$. There are several other characterizations of infinite fibers see \cite[Proposition 2.4]{Delgmanuscripta1} and \cite[Lemma 1.8]{delgadogorenstein} for further details.

Many of the above properties hold for fractional ideals of $\mathcal{O}$ and its set of values. For the  convenience of the reader, we present them in the next section.
\medskip

\subsection{Basic properties of value set of fractional ideals}

A fractional ideal $J\subset \mathcal{K}$ of $\mathcal{O}$ is a $\mathcal{O}$-module such that there exists a regular element $h\in\mathcal{O}$ satisfying $hJ\subset\mathcal{O}$. It follows that for any fractional ideal $J$ of \(\mathcal{O}\) its value set \(E:=\vu (J)\) is a relative ideal of \(S,\) i.e. \(S+E\subseteq E\) and there exists \(\gamma\in S\) such that \(\gamma+E\subseteq S.\) 

By \cite{Danna1,Danna2}, for any two fractional ideals \(J_2\subset J_1\) of $\mathcal{O}$ we can compute the length \(l(J_1/J_2)\) (as $\mathcal{O}$-modules) by comparing their value sets by means a saturated chain as follows. 

For a fractional ideal \(J\subset \mathcal{K},\) a saturated chain in \(E=\vu (J)\) is a sequence
\[\alpha^0<\alpha^1<\cdots<\alpha^n\]
of elements in \(E\) such that for every element \(\epsilon\in\mathbb{Z}^r\) such that \(\alpha^i<\epsilon<\alpha^{i+1}\) one has \(\epsilon\notin E.\) Such a chain is said to have length \(n.\) According to \cite[Proposition 2.3]{Danna1}, 
any saturated chain in \(E\) between \(\alpha^0\) and \(\alpha^n\) has the same length. This property allows us to define a distance function between two elements in $E$: if \(\alpha^1,\alpha^n\in E\) with \(\alpha^1<\alpha^n\) then its distance in \(E\), denoted by \(d_{E}(\alpha^1,\alpha^n),\) is the length of any saturated chain in \(E\) with \(\alpha^1\) as initial element and \(\alpha^n\) as final element. 
\medskip

As for \(S,\) the value set of a fractional ideal also satisfies properties \((1),(2)\) and \((3)\) (see \cite{Danna1}). So, if \(J\) is a fractional ideal of $\mathcal{O}$, then its value set \(E\) has always a minimum \(m_E:=\min\{\alpha\in E\}\) and a conductor \(\con_E:=\min\{\gamma\in E\;\colon\;\gamma+\mathbb{N}^r\subseteq E\}.\) 

In this way, the colength of fractional ideals can be computed according the following 
\begin{theorem}\cite[Section 2]{Danna1}\label{defin:distance}
    Let \(J_2\subset J_1\) be fractional ideals of \(\mathcal{O}\) with \(E_i=\vu(J_i)\) for $i=1,2$. Then,
    \[l(J_1/J_2)=d(E_1\setminus E_2):=d_{E_1}(m_{E_1},\con_{E_2})-d_{E_2}(m_{E_2},\con_{E_2}).\]    
\end{theorem}

\begin{rem}
    In general, for any two subsets \(E_1,E_2\subseteq\mathbb{Z}^r,\) \(E_2\subsetneq E_1\) satisfying properties \((1),(2),(3)\) we define its distance as \(d(E_1\setminus E_2)=d_{E_1}(m_{E_1},\con_{E_2})-d_{E_2}(m_{E_2},\con_{E_2}).\)

    Also, it is obvious that for any \(\gamma\geq \con_{E_2}\) we have 
    \[d(E_1\setminus E_2)=d_{E_1}(m_{E_1},\gamma)-d_{E_2}(m_{E_2},\gamma).\]
\end{rem}

This method has the disadvantage that we need a lot of information about the value set \(E.\) In order to avoid the use of a saturated chain in \(E,\)  Guzmán and Hefez in \cite{Hefezcolength} provided an alternative method to compute colengths just by using the set of relative maximal points of the value set and its projections. 

\begin{rem}
    The notion of maximal, relative maximal and absolute maximal for a value set is defined in an analogous way as in the semigroup case. For any value set \(E\), we will denote by \(M(E), RM(E), AM(E)\) the sets of maximal, relative maximals and absolute maximals of \(E.\)
\end{rem}

Let us briefly explain the Guzm\'an and Hefez's method to compute colengths without the use of a saturated chain, we refer to \cite{Hefezcolength} for further details. 

For any fractional ideal \(J\) of \(\mathcal{O},\) there is a canonical filtration indexed by \(\alpha\in\mathbb{Z}^r\) defined as 
\[J(\alpha)=\{h\in J\; ;\; \vu(h)\geq \alpha\}.\]
Therefore, given \(J_2\subseteq J_1\) two fractional ideals with value sets \(E_i=\vu(J_i)\), we have that for any \(\gamma\geq \con_{E_2}\) the colength (and hence its distance) is 
\begin{equation}\label{codim}l(J_1/J_2)=l(J_1/J_1(\gamma))-l(J_2/J_2(\gamma)).\end{equation}
Recall that $J_1(\gamma)=J_2(\gamma)$. In this way, we have 
\begin{theorem}\cite[Cor. 11]{Hefezcolength}\label{thm:colengthcalculation}
Let \(J\) be a fractional ideal of \(\mathcal{O}\) with value set \(E=\vu(J)\) and $m_E=\alpha^0$. If \(\gamma\geq \con_E,\) then
\[l\left (\frac{J}{J(\gamma)}\right )=\sum_{i=1}^{r}(\gamma_i-\alpha^0_i-\sharp ((\mathbb{N}+\alpha^0_i)\setminus pr_i(E))-\Theta_i),\]
where \(\Theta_i\) is defined as 
\[\Theta_1=0,\quad \Theta_i=\sharp\bigcup_{\{i\}\subsetneq \jnd\subseteq \{1,\dots,i\}}pr_*(RM(E_{\jnd})),\quad \text{for}\; 2\leq i\leq r,\]
with $pr_*(\alpha_{j_1},\ldots ,\alpha_{j_s})=\alpha_{j_s}$ and $E_{\jnd}=pr_{\jnd}(E)$ if $\jnd=\{j_1,\ldots ,j_s\}$.
\end{theorem} 
 
As an application of Theorem \ref{thm:colengthcalculation} we will provide an alternative way to compute the delta invariant $\delta(C)$ of a reduced plane curve $C$. 

\begin{ex}\label{delta} Let \(C=\cup_{i\in \ind}C_i\) be a reduced plane curve such that each branch $C_i$ is defined by $f_i$. If $\mathcal{O}$ denotes its local ring, then its normalization \(\overline{\mathcal{O}}\) is isomorphic to $\mathbb{C}\{t_1\}\times\cdots \times\mathbb{C}\{t_r\}$. Since \(\mathcal{O}\) is a Gorenstein ring (see \cite{delgadogorenstein}),  the conductor ideal \(\mathcal{C}:=\{z\in\overline{\mathcal{O}}\;\colon\;z\overline{\mathcal{O}}\subset\mathcal{O}\}\) is such that
\[\delta(C)=l(\frac{\overline{\mathcal{O}}}{\mathcal{O}})=l(\frac{\mathcal{O}}{\mathcal{C}})=\frac{1}{2}l(\frac{\overline{\mathcal{O}}}{\mathcal{C}}).\]

 Notice that \(\vu(\overline{\mathcal{O}})=\mathbb{N}^r\) and \(\vu(\mathcal{C})=\con_S+\mathbb{N}^r\), in particular, these values sets do not have maximal points, that is, \(\Theta_i(\vu(\overline{\mathcal{O}}))=\Theta_i(\vu(\mathcal{C}))=0\) for all \(i\in\ind\).  In addition, we get $$m_{\vu(\overline{\mathcal{O}})}=\con_{\vu(\overline{\mathcal{O}})}=(0,\ldots ,0)\in\mathbb{N}^r\ \ \mbox{and}\ \ 
m_{\vu(\mathcal{C})}=\con_{\vu(\mathcal{C})}=\con_S.$$

According to (\ref{codim}) we get $l(\frac{\overline{\mathcal{O}}}{\mathcal{C}})=l(\frac{\overline{\mathcal{O}}}{\overline{\mathcal{O}}(\con_S)})-l(\frac{\mathcal{C}}{\mathcal{C}(\con_S)})$. Notice that $\mathcal{C}(\con_S)=\mathcal{C}$, that is $l(\frac{\mathcal{C}}{\mathcal{C}(\con_S)})=0$ and, by Theorem \ref{thm:colengthcalculation} we have
$$l(\frac{\overline{\mathcal{O}}}{\overline{\mathcal{O}}(\con_S)})=\sum_{i\in\ind}(pr_i(\con_S)-0-\sharp(\mathbb{N}+0)\setminus pr_i(\mathbb{N}^r))-0))=\sum_{i\in\ind}pr_i(\con_S)=\sum_{i\in\ind}(c_i+\sum_{j\neq i}I_{i,j}),$$
where the last equality follows by (\ref{conductor}).

Since \(c_i=2\delta_i\) where \(\delta_i=\delta(C_i)\) is the delta invariant of the branch we get 
\[\delta(C)=\frac{1}{2}l \left(\frac{\overline{\mathcal{O}}}{\mathcal{C}}\right)=\frac{1}{2}l \left(\frac{\overline{\mathcal{O}}}{\overline{\mathcal{O}}(\con_S)}\right)=\frac{1}{2}\sum_{i\in\ind}(c_i+\sum_{j\neq i}I_{i,j})=\sum_{i\in\ind}(\delta_i+\sum_{j<i}I_{i,j}).\]
In this way, we recover the well-known formula for the delta invariant of the curve $C$ and, as an immediate consequence of Proposition \ref{defin:distance} we also obtain that 
\[\delta(C)=\frac{1}{2}l(\frac{\overline{\mathcal{O}}}{\mathcal{C}})=l(\frac{\mathcal{O}}{\mathcal{C}})=d(\mathcal{O}\setminus\mathcal{C})=d_S(0,\con_S)-d_{\vu(\mathcal{C})}(\con_S,\con_S)=d_S(0,\con_S).\]
\end{ex}

In \cite{hefezhernandes24} is presented a way to compute the data $\Theta_i$ given in Theorem \ref{thm:colengthcalculation} without knowing the relative maximal of the values set of a fractional ideal or any information of the values set of $E$. Let us present this result.

Given $J\subset \mathcal{K}$ a fractional ideal of $\mathcal{O}$ with $E:=\vu(J)$ and $\pi:\mathcal{K}\rightarrow\mathcal{K}_i$ the natural projection we put
$$\mathcal{N}_i(J):=J\cap\ker\pi_i\ \ \ \mbox{and}\ \ \ \mathcal{N}_{\jnd}(J):=\cap_{j\in\jnd}\mathcal{N}_j(J).$$
In this way, setting $[1,i)=\{1,\ldots ,i-1\}$ for $1<i\leq r$ we get (cf. \cite[Cor. 2.8]{hefezhernandes24}) 
\begin{equation}\label{thetai}
\Theta_i=\sharp (\pr_i(E)\setminus v_i(\mathcal{N}_{[1,i)}(J))).\end{equation}
We consider $\mathcal{N}_{[1,1)}(J)=pr_i(E)$ to obtain $\Theta_1=0$ as in Theorem \ref{thm:colengthcalculation}. In this way, we get the following

\begin{theorem}\cite[Cor. 2.9]{hefezhernandes24}\label{thm:colengthannihilator}
  Let \(J_2\subset J_1\) be fractional ideals of \(\mathcal{O}\) with \(E_i=\vu(J_i)\) for $i=1,2$. Then,
\[l\left (\frac{J_1}{J_2}\right )=
\sum_{i\in\ind}\left ( \sharp(pr_i(E_1)\setminus pr_i(E_2))-\sharp(pr_i(E_1)\setminus v_i(\mathcal{N}_{[1,i)}(J_1)))+
\sharp(pr_i(E_2)\setminus v_i(\mathcal{N}_{[1,i)}(J_2))) \right ).\]
\end{theorem} 

Notice that Theorem \ref{defin:distance}, Theorem \ref{thm:colengthcalculation} (and (\ref{codim})) and Theorem \ref{thm:colengthannihilator} provide alternative ways to compute colength of fractional ideals according to the data available in each situation. 

In this paper, we are interested in computing the codimension of a particular fractional ideal that gives us an important analytic invariant of a plane curve: the Tjurina number.

\subsection{Tjurina number}\label{subsec:tjurinaBerger}
Let $C$ be a plane curve defined by $f=\prod_{i\in\ind}f_i$ and $C_i$ be the branch given by $f_i=0$. The Tjurina number of $C$ is $\tau=\tau(f):=\dim_{\mathbb{C}}\mathbb{C}\{x,y\}/(f,f_x,f_y)$. Denoting by $\overline{h}$ the class of $h\in\mathbb{C}\{x,y\}$ in $\mathcal{O}$ and  considering the ideal $\mathbf{J}:=\mathcal{O}\overline{f_x}+\mathcal{O}\overline{f_y}$ we have that $\tau=l(\frac{\mathcal{O}}{\mathbf{J}})$.

An alternative, but equivalent, approach to compute the Tjurina number is by using the module of K\"{a}hler differentials of the curve. Let \(\Omega^1=\mathbb{C}\{x,y\}\textup{d}x+\mathbb{C}\{x,y\}\textup{d}y\) be the \(\mathbb{C}\{x,y\}\)-module of holomorphic forms on \(\mathbb{C}^2\) and consider the submodule $\mathcal{F}(f):=\mathbb{C}\{x,y\}\textup{d}f+f\Omega^1$. 
The module of K\"{a}hler differentials of $C$ is
$$\Omega_f:=\frac{\Omega^1}{\mathcal{F}(f)},$$
that is the $\mathcal{O}$-module $\mathcal{O}\textup{d}x+\mathcal{O}\textup{d}y$ module the relation $\textup{d}f=0$.

If \(\varphi_i=(x_i,y_i)\in\mathbb{C}\{t_i\}\times \mathbb{C}\{t_i\}\) is a parameterization (non necessarily a Newton-Puiseux parameterization) of the branch \(C_i\) and $h(x,y)\in\mathcal{O}$ then we denote $\varphi_i^*(h):=h(x_i,y_i)\in\mathbb{C}\{t_i\}$. In addition, given \(\omega=A(x,y)\textup{d}x+B(x,y)\textup{d}y\in\Omega_f,\) we define 
\[\varphi_i^\ast(\omega)\colon=t_i(\varphi_i(A)\cdot x_i'+\varphi_i(B)\cdot y_i')\in\mathbb{C}\{t_i\},\]
where \(x_i',y_i'\) denote, respectively, the derivative of \(x_i, y_i\in\mathbb{C}\{t_i\}\) with respect to \(t_i.\) We put 

$$\varphi^*(\Omega_f)=\{(\varphi_1^\ast(\omega),\ldots ,\varphi_r^\ast(\omega)) : \omega\in\Omega_f\}\subset \mathcal{K}.$$ 

By \cite[Theorem 3]{bayeretal}, if $Tor(\Omega_f)$ denotes the torsion submodule of $\Omega_f$ then we have $\ker (\varphi^*)=Tor(\Omega_f)$ and the following $\mathcal{O}$-module isomorphism:
\begin{equation}\label{quot-torsion}
\varphi^*(\Omega_f)\cong \frac{\Omega_f}{Tor(\Omega_f)}.
\end{equation}
In this way, $\varphi^*(\Omega_f)$ is a fractional ideal of $\mathcal{O}$ and, considering $v_i(\omega):=v_i(\varphi_i^*(\omega))$, its value set is given by 

\[\Lambda_f=\vu(\varphi^\ast(\Omega_f))=\{\vu(\omega)\colon=(v_1(\omega),\dots,v_r(\omega));\;\;\omega\in \Omega_f\}.\]
Recall that $S\setminus\{0\}\subset\Lambda_f$ so, $\con_{\Lambda_f}\leq\con_S$ and $m_{\Lambda_f}=(\beta^1_0,\ldots ,\beta^r_0)$ where $\beta^i_0=\min\{\alpha\in S(f_i)\setminus\{0\}\}$ is the multiplicity of the branch $C_i$.

The set $\Lambda_f$ is an important analytic invariant of the curve which was used for the analytical classification of plane curves (see \cite{HernandesRodrigues}) and, according to Pol \cite[Proposition 3.31]{PolLogarithmic}, the Jacobian ideal $\mathbf{J}$ is isomorphic to $\varphi^*(\Omega_f)$. Moreover, she shows that 
\begin{equation}\label{relation1}
\vu(\overline{Af_x+Bf_y})=\vu(\overline{A}\textup{d}y-\overline{B}\textup{d}x)+\con_S-\underline{1},\ \ \ \mbox{consequently}\ \ \ \vu(\mathbf{J})=\Lambda_f+\con_S-\underline{1}
\end{equation}
where $\underline{1}=(1,\ldots ,1)\in\mathbb{N}^r$.

In the case \(r=1,\) Berger \cite{Berger63} proved that the Milnor number and the Tjurina number of $C_f$ are related by \(\tau=\mu-\sharp(\Lambda_f\setminus S).\) Since the distance function is the natural generalization for the difference of values set of fractional ideals in the irreducible case, it is natural to ask for a extension of Berger's expression for the case \(r>1.\) From the identity \eqref{relation1} and the previous results, we obtain the following generalization of Berger's result.

\begin{theorem}\label{thm:Tjurinaformula}
    Let \(C\) be a reduced plane curve. With the previous notation, let us denote by \(\overline{\Lambda}=\Lambda_f\cup\{0\}.\) Then, we have 
    \begin{equation}
        \tau=\mu-d(\overline{\Lambda}\setminus S).
    \end{equation}
\end{theorem}
\begin{proof}
Consider $C=\cup_{i=1}^r C_i$. Since $\mathbf{J}\subset\mathcal{O}$ we have \(E:=\vu(\mathbf{J})\subset S\) and \(\con_E\leq m_E+\con_S.\) As we have remarked, $m_{\Lambda}=(\beta^1_0,\ldots ,\beta^r_0)$ and by (\ref{relation1}) we get $m_E=\con_S+(\beta^1_0,\ldots ,\beta^r_0)-\underline{1}$. Thus, considering the relation (\ref{codim}) with \(\gamma=2\con_S+(\beta_0^1,\dots,\beta_0^r)-\underline{1}\) we have 
\begin{equation}\label{equat1}\tau=l \left(\frac{\mathcal{O}}{\mathbf{J}}\right)=l \left(\frac{\mathcal{O}}{\mathcal{O}(\gamma)}\right)-l \left(\frac{\mathbf{J}}{\mathbf{J}(\gamma)}\right).\end{equation}

Notice that $l(\mathcal{O}/\mathcal{O}(\gamma))=l(\mathcal{O}/\mathcal{O}(\con_S))+l(\mathcal{O}(\con_S)/\mathcal{O}(\gamma))$ with $\vu(\mathcal{O}(\con_S))=\con_S+\mathbb{N}^r$ and $\vu(\mathcal{O}(\gamma))=\gamma+\mathbb{N}^r$. In this way,
by Theorem \ref{defin:distance} and Theorem \ref{thm:colengthcalculation},  we have
\begin{equation}\label{equat2}
\begin{array}{ll}
l \left(\frac{\mathcal{O}}{\mathcal{O}(\gamma)}\right)=l \left(\frac{\mathcal{O}}{\mathcal{O}(\con_S)}\right)+l \left(\frac{\mathcal{O}(\con_S)}{\mathcal{O}(\gamma)}\right) & =d_S(0,\con_S)+\sum_{i=1}^r (pr_i(\con_S)+\beta_0^i-1)\vspace{0.2cm} \\
& =d_S(0,\con_S)+\sum_{i=1}^{r}\beta_0^i+\mu-\underline{1}
\end{array}\end{equation}
where the last equality follows by Remark \ref{defin:distance}.

On the other hand, the relation (\ref{relation1}) between \(\vu(\mathbf{J})\) and \(\Lambda_f\) gives
\[l\left (\frac{\mathbf{J}}{\mathbf{J}(\gamma)}\right )=l\left (\frac{\varphi^*(\Omega_f)}{\varphi^*(\Omega)(\con_S+(\beta_0^1,\ldots ,\beta_0^r))}\right )=l\left (\frac{\varphi^*(\Omega_f)}{\varphi^*(\Omega)(\con_S)}\right )+l\left (\frac{\varphi^*(\Omega_f)(\con_S)}{\varphi^*(\Omega)(\con_S+(\beta_0^1,\ldots ,\beta_0^r))}\right ).\]
Since $\con_{\Lambda_f}\leq\con_S$, we get $\vu(\varphi^*(\Omega_f)(\con_S+(\beta_0^1,\ldots ,\beta_0^r))=\con_S+(\beta_0^1,\ldots ,\beta_0^r)+\mathbb{N}^r$ and $\vu(\varphi^*(\Omega_f)(\con_S))=\con_S+\mathbb{N}^r$, that is, the second above summand is $\sum_{i=1}^{r}\beta^i_0$. In addition, denoting $\overline{\Lambda}=\Lambda_f\cup\{0\}$ we have
$$l\left (\frac{\varphi^*(\Omega_f)}{\varphi^*(\Omega)(\con_S)}\right )=d_{\Lambda_f}((\beta_0^1,\ldots ,\beta_0^r),\con_S)=d_{\overline{\Lambda}}(0,\con_S)-\underline{1}.$$ In this way, we get \begin{equation}\label{equat3}l
\left(\frac{\mathbf{J}}{\mathbf{J}(\gamma)}\right)=d_{\overline{\Lambda}}(0,\con_S)-\underline{1}+\sum_{i=1}^{r}\beta^i_0.
\end{equation}
Therefore, by (\ref{equat1}), (\ref{equat2}) and (\ref{equat3}), we obtain
\[\tau=\mu+d_S(0,\con_S)-d_{\overline{\Lambda}}(0,\con_S)=\mu-d(\overline{\Lambda}\setminus S).\]
\end{proof}


\section{Logarithmic differentials}\label{sec:logdif}

As before, let  \(C\) be a reduced plane curve defined by \(f\). According to (\ref{relation1}), the values set $\vu(\mathbf{J})$ determines and it is determined by the set $\Lambda$. In \cite{PolLogarithmic}, Pol shows that such analytical invariants are related to the values set of residues of logarithmic differentials or equivalently to the values set of the Saito module.
Let us recall these objects and some results concerning to them.

According to Saito \cite{SaitoLogari}, a meromorphic differential \(W\in (1/f)\Omega^1\) is a logarithmic form along \(C\) if there exist \(\eta\in \Omega^1,\) \(p,q\in\mathbb{C}\{x,y\}\) with \(gcd(q,f)=1\) such that \(qW=(p\textup{d}f+f\eta)/f\) or equivalently, $\frac{W\wedge \textup{d}f}{\textup{d}x\wedge\textup{d}y}\in\mathbb{C}\{x,y\}$ and he denotes the set of logarithmic forms along $C$ by $\Omega^1(\log\ C)$. 

Since $W=\omega/f\in \Omega^1(\log\ C)$ is equivalent to get $q\omega\in\mathcal{F}(f)=\mathbb{C}\{x,y\}\textup{d}f+f\Omega^1$ for some $q\in\mathbb{C}\{x,y\}$ coprime with $f$ or $P_f(\omega):=\frac{\omega\wedge\textup{d}f}{\textup{d}x\wedge\textup{d}y}\in (f)$ we can consider the $\mathbb{C}\{x,y\}$-module
\begin{equation}\label{log-dif}
\begin{array}{ll}f\cdot\Omega^1(\log\ C) & =\left \{\omega\in\Omega^1;\ \exists\ q\in\mathbb{C}\{x,y\}, gcd(q,f)=1\ \mbox{such that}\ q\omega\in\mathcal{F}(f)\right \}\vspace{0.2cm}\\
& =\left \{\omega\in\Omega^1;\ P_f(\omega)=\frac{\omega\wedge \textup{d}f}{\textup{d}x\wedge \textup{d}x}\in (f)\right \},\end{array}\end{equation}
called the Saito module associated to \(C\). 

We have that $f\cdot\Omega^1(\log\ C)$ is generated by two elements. Moreover,

\vspace{0.2cm}
{\bf Saito's criterion:} $\{\omega_1,\omega_2\}$ is a set of generators for $f\cdot\Omega^1(\log C)$ if and only if $\frac{\omega_1\wedge\omega_2}{\textup{d}x\wedge\textup{d}y}=uf$ where $u\in\mathbb{C}\{x,y\}$ is a unit.
\vspace{0.2cm}

Notice that $\mathcal{F}(f)\subset f\cdot\Omega^1(\log\ C)$, by \cite[Proposition 3.22]{PolLogarithmic}, we have that $Tor(\Omega_f)$ is isomorphic (as $\mathcal{O}$-module) to $f\cdot\Omega^1(\log\ C)/\mathcal{F}(f)$. In particular, by \cite[Proposition 3.31]{PolLogarithmic} and (\ref{quot-torsion}), we get 
$$\mathbf{J}\cong \varphi^*(\Omega_f)\cong\frac{\Omega^1}{f\cdot\Omega^1(\log\ C)}.$$

Given \(\omega\in f\cdot\Omega^1(\log\ C),\) such that $q\omega=p\textup{d}f+f\eta$ where $\eta\in\Omega^1, p,q\in\mathbb{C}\{x,y\}$ with $gcd(q,f)=1$ the residue of $\omega$ is \(res(\omega)=\overline{p}/\overline{q}\in\mathcal{K}\) where $\overline{h}$ denotes the class of $h\in\mathbb{C}\{x,y\}$ in $\mathcal{O}$. The $\mathcal{O}$-module of logarithmic residues along \(C\) is then defined as 
\[\mathcal{R}_C:=\{res(\omega);\; \omega\in f\cdot\Omega^1(\log\ C)\}\subset\mathcal{K}.\]

We have that $\mathcal{R}_C$ is a fractional ideal of $\mathcal{O}$ and its values set $\Delta_f:=\vu(\mathcal{R}_C)$ satisfies (cf. \cite[Cor. 3.32]{PolLogarithmic})
\begin{equation}\label{delta-lambda}
\alpha\in \Delta_f\Leftrightarrow F(-\alpha,\Lambda_f)=\emptyset.\end{equation}

According to Pol \cite{PolLogarithmic} we have that \(\con(\Delta_f)\) is \(-(\obeta_0^{(1)},\dots,\obeta_0^{(r)})+(1,\dots,1).\) 

\begin{rem}\label{rem-inf-fiber}
If $\ind=\{1,\ldots ,r\}$ with $r>1$ and $C=\cup_{i\in\ind}C_i$ is a plane curve where each branch $C_i$ is defined by $f_i$ then for \(\emptyset\neq \jnd\subsetneq \ind\) we denote $f_{\jnd}=\prod_{j\in\jnd}f_j$. In this case, given 
\(\omega\in f_\jnd\cdot \Omega^1(\log\ C_\jnd)\) 
 whose class module $\mathcal{F}(f)$ is $\overline{\omega}$, we get $v_j(\overline{\omega})=\infty$
for all \(j\in\jnd\). Moreover,
$$\omega\in f_\jnd\cdot \Omega^1(\log\ C_\jnd)\setminus f_{\ind\setminus\jnd}\cdot\Omega^1(\log\ C_{\ind\setminus\jnd})\ \ \ \mbox{if and only if}\ \ \ F^{\Lambda}_{\ind\setminus\jnd}(\vu(\overline{\omega}))=\infty,$$ that is $\vu(\overline{\omega})$ has infinite fiber in $\Lambda=\Lambda_f$ (see Definition \ref{inf-fiber}).
\end{rem}

\subsection{Infinite fibers of diagonal curve with two branches}\label{fiber-diag}

In what follows, we will use the values set of logarithmic residues along $C$ to obtain information on the set $\Lambda$, as the infinite fibers and its conductor, when $C$ is defined by $f=f_1f_2$ and it has two equisingular branches with values semigroup $S(C_1)=S(C_2)=\langle \overline{\beta}_0,\ldots,\overline{\beta}_g\rangle$ and intersection multiplicity 
$I=[f_1,f_2]> n_g\overline{\beta}_g$ that we call {\it diagonal curve with two branches} (see \cite{Delgadoari}).

Before that, let us recall some important fact about irreducible curves. Let $C$ be an irreducible plane curve with values semigroup \(S=\langle\obeta_0,\dots,\obeta_g\rangle\). Without loss of generality, we can assume that $C$ is defined by a Weierstrass polynomial $f\in\mathbb{C}\{x\}[y]$ with $\deg_y(f)=\beta_0$ and Newton-Puiseux parametrization $\varphi$ as (\ref{puiseux-param}). 
\medskip

Given $n\in\mathbb{N}\setminus\{0\}$ we consider the $\mathbb{C}$-vector space 
\begin{equation}\label{Pn}
P_n=\{h\in\mathbb{C}\{x\}[y];\ \deg_y(h)<n\}.
\end{equation} Since $f\in\mathbb{C}\{x\}[y]$ is a Weierstrass polynomial with $\deg_y(f)=\beta_0$, any $H\in\mathbb{C}\{x,y\}$ can be uniquely expressed (by the Weierstrass division theorem) as $H=qf+h$ with $q\in\mathbb{C}\{x,y\}$ and $h\in P_{\beta_0}$. So, $\mathbb{C}\{x,y\}=P_{\beta_0}\oplus f\cdot\mathbb{C}\{x,y\}$ and the classes of $H$ and $h$ in $\mathcal{O}=\frac{\mathbb{C}\{x,y\}}{(f)}$ are equal. 

If $v_f$ indicates the discrete valuation associated to $C$ then we put $v_f(H):=v_f(\overline{H})$ where $\overline{H}$ denotes the class of $H\in\mathbb{C}\{x,y\}$ in $\mathcal{O}$. In this way, $S=v_f(\mathcal{O})=v_f(\mathbb{C}\{x,y\})=v_f(P_{\beta_0})$. Following \cite[Sec. 3]{hernandessemiroots}, we introduce the following \(\mathbb{C}\)-vector spaces:
\[
\mathcal{E}(f) = P_{\beta_0} \, \textup{d}x + P_{\beta_0-1} \, \textup{d}y,\quad  \mathcal{G}(f) = \mathbb{C}\{x,y\} \, \textup{d}f + \mathbb{C}\{x,y\} \, f \, \textup{d}x, 
\]
where $P_{\beta_0}$ is given in (\ref{Pn}).
With the above notation, we have \cite[Lemma 3.3]{hernandessemiroots}
\[
\Omega^1 = \mathcal{E}(f) \oplus \mathcal{G}(f).
\]
Since $\mathcal{G}(f)\subset\mathcal{F}(f)$, if $\omega=\omega_0+\omega_1\in\Omega^1$ with $\omega_0\in\mathcal{E}(f)$ and $\omega_1\in\mathcal{G}(f)$ then the classes $\overline{\omega}$ and $\overline{\omega_0}$ of $\omega$, respectively of $\omega_0$, in $\Omega_f$ are equal and $v_f(\overline{\omega})=v_f(\overline{\omega_0})$. In this way, we get $\Lambda_f=\{v_f(\overline{\omega});\ \omega\in\mathcal{E}(f)\}$. In what follows, to simplify the notation, given $\omega\in\Omega^1$ when we put $v_f(\omega)$ we understand $v_f(\overline{\omega})$ where $\overline{\omega}$ indicates the class of $\omega$ module $\mathcal{F}(f)$. 
\medskip

Let us now consider \(\omega = A \, \textup{d}x - B \, \textup{d}y \in f \cdot \Omega(\log\ C_f)\). According to (\ref{log-dif}) we get  \(A f_y + B f_x =  M f\) for some \(M \in \mathbb{C}\{x,y\}\). In particular, the relations
\begin{equation}\label{log-M}
f_x \omega = A \, \textup{d}f - f M \, \textup{d}y \quad \text{and} \quad f_y \omega = -B \, \textup{d}f + f M \, \textup{d}x,
\end{equation}
allow us to compute the residue of \(\omega\) as \(\text{res}(\omega) = \frac{\overline{A}}{\overline{f_x}} = \frac{\overline{B}}{\overline{f_y}}\). Thus, the value set of logarithmic residues along \(C_f\) can be done by  
\begin{equation}\label{Deltai}
\begin{array}{ll}\Delta_f & =\{v_f(res(\omega))=v_f(B)-v_f(f_y);\; Adx+Bdy\in f\cdot \Omega(\log\ C_f)\} \\
& = \{v_f(B)-(\mu+\obeta_0-1);\; Adx+Bdy\in f\cdot \Omega(\log\ C_f)\},
\end{array}
\end{equation}
where $v_f(f_y)=\mu+\obeta_0-1$ (see (\ref{fy-cond})).
In particular, as we are considering the irreducible case, the property (\ref{delta-lambda}) translates to the fact that 
\begin{equation}\label{relation}
\lambda\in\Lambda_f\ \ \ \Leftrightarrow\ \ \ -\lambda\notin \Delta_f.
\end{equation}

\begin{proposition} \cite[Prop. 3.7]{hernandessemiroots}\label{prop:MBirreducible}
    Let \(\omega = A \, \textup{d}x - B \, \textup{d}y \in \mathcal{E}(f) \cap f \cdot \Omega^1(\log\ C_f),\) i.e. \( B f_x + A f_y = M f.\) Under the previous notation, we have
\[
\varphi^*(M) = \frac{e_k \obeta_{k+1}}b \, t^{v_f(M)} + (\text{h.o.t.}),\quad
\varphi^*(B) = b \, t^{v_f(B)} + (\text{h.o.t.}),
\]
with \(b\in\mathbb{C}^{\ast}\) and \(\nu_f(B) = \nu_f(M) + \obeta_0\), where 
\(
k = \max_{0\leq i<g} \{ i;\  e_i \nmid v_f(res(\omega)) \}.
\)
\end{proposition}

In what follows, let $C_f=C_1\cup C_2$ where $C_1$ and $C_2$ are two equisingular plane branches defined by Weierstrass polynomials $f_1, f_2\in\mathbb{C}\{x\}[y]$ with value semigroup $S_1=S_2=\langle \obeta_0,\ldots ,\obeta_g\rangle$ and intersection multiplicity \(I:=[f_1,f_2]_0\geq n_g\obeta_g,\) where $n_g=e_{g-1}=\gcd(\obeta_0,\ldots ,\obeta_{g-1})$. In particular, there exists a Newton-Puiseux parametrization $\varphi_i(t_i)=(t_i^{\obeta_0},\sum_{k\geq\obeta_0}a^{(i)}_{k}t_i^k)$ for $C_i$ and $i=1,2$ such that 
\begin{equation}\label{equal-coeficient}
\begin{array}{c}
a^{(1)}_{k}=a^{(2)}_{k}\ \mbox{for every}\ k\leq\beta_g\ \mbox{if}\ I>n_g\obeta_g\\
a^{(1)}_{k}=a^{(2)}_{k}\ \mbox{for every}\ k<\beta_g\ \mbox{if}\ I=n_g\obeta_g.\end{array}
\end{equation}

We would like to use Proposition \ref{prop:MBirreducible} in the context of a plane curve with two equisingular branches with \(I\geq n_g\obeta_g.\) To do so, we need the following technical lemmas to generalize Proposition \ref{prop:MBirreducible} in this context.
\medskip
As condition \(I\geq n_g\obeta_g\) implies that both branches have the same value semigroup, let us recall a few facts regarding the maximal contact curves associated to a branch. For a fixed \(i\in\{1,2\}\), let say \(i=1\), any \(h\in\mathbb{C}\{x,y\}\) satisfying \([f_1,h]=\obeta_q\) will be called a maximal contact curve of genus \(q-1\) with \(f_1\).  In \cite{zariski}, Zariski considers a particular set of maximal contact curves $D_q$ for $1\leq q\leq g$ defined by a truncation of the Puiseux series of the curve $C_1$ as follows: 
\[
s^{(1)}_q(x)=\sum_{\tiny\begin{array}{c}
		j\in (\beta_0)\\\ \beta_0\leq j<\beta_1
\end{array}}a^{(1)}_jx^{j/\beta_0}+\cdots+\sum_{\tiny\begin{array}{c}
		j\in (e_{q-1})\\\beta_{q-1}\leq j<\beta_q
\end{array}}a^{(1)}_jx^{j/\beta_0}.
\]

The minimal polynomial $\frakg_q(x,y)\in\mathbb{C}[x,y]$ of $s_q(x)$ is given by

\begin{equation}\label{gq}
\frakg_q(x,y)=\prod_{\epsilon^{\beta_0/e_{q-1}}=1}(y-s^{(1)}_q(\epsilon^{\beta_0}x))
\end{equation}
and the plane branch $D_q$ given by $\frakg_q$ has maximal contact with $C_1$. In addition, $\frakg_q$ is monic with $\deg_y(\frakg_q)=\frac{\overline{\beta_0}}{e_{q-1}}$.
In that follows we put $\frakg_{0}=x$. The branches \(D_q\) are called 
semiroots of \(C_1\) and 
any element \(h\in\mathbb{C}\{x\}[y]\) with $\deg_y(h)<\deg_y(f_1)$, that is, $h\in P_{\beta_0}$, admits a unique \(\mathit{G}\)--adic expansion (cf.  \cite[Chap. 1 and 3]{Abexpansion}) in terms 
of $\mathit{G}=\{\frakg_{0},\frakg_1,\ldots ,\frakg_{g}\}$, that is,

\begin{equation}\label{expansionh} h = \sum_{\alpha = (\alpha_0, \dots, \alpha_g) \in \mathbb{N}^{g+1}} b_{\alpha}\frakg_0^{\alpha_0} \frakg_1^{\alpha_1} \cdots \frakg_{g-1}^{\alpha_{g-1}} \frakg_g^{\alpha_g}, \end{equation}
with $b_{\alpha}\in\mathbb{C}$ and

 \begin{enumerate}
     \item   \( 0 \leq \alpha_k < n_{k}=\frac{e_{k-1}}{e_{k}} \) for \( 1 \leq k \leq g \).
     \item If we denote \(h_\alpha=\frakg_0^{\alpha_0} \frakg_1^{\alpha_1} \cdots \frakg_{g-1}^{\alpha_{g-1}} \frakg_g^{\alpha_g}\), then \(v_1(h_\alpha)\neq v_1(h_\gamma)\) for \(\alpha\neq\gamma.\)
 \end{enumerate}
 
 \begin{rem}\label{qh}
Notice that if $h\in\mathbb{C}\{x\}[y]$ is given as (\ref{expansionh}), that is 
\[h=\sum_{\gamma=(\gamma_0,\ldots ,\gamma_g)\in\mathbb{N}^{g+1}}b_{\gamma}\frakg_0^{\gamma_0}\frakg_1^{\gamma_1}\cdots\frakg_{g-1}^{\gamma_{g-1}}\frakg_g^{\gamma_g}\]
and \(v_1(h)=\alpha_0\obeta_0+\dots+\alpha_q\obeta_q\) with \(\alpha_q \neq 0\) for $q\leq g$ and \(0\leq \alpha_k\leq n_k\) for \(1\leq k\leq q\), then 
\[\varphi_1^*(h)=\varphi_1^*(b_{\alpha}\frakg_0^{\alpha_0} \cdots \frakg_q^{\alpha_q})+(\text{h.o.t.})= Q^h_{v_1(h)}(a^{(1)})t^{v_f(h)}+(\text{h.o.t.}), \]
where \(Q^h_{v_1(h)}(a^{(1)})\) is a polynomial in the coefficients of the Puiseux series of \(C_1.\) 
 The polynomial \(Q^h_{v_1(h)}(a^{(1)})\) can be explicitly computed and satisfy some useful properties related to the coefficients of the Puiseux series.  We refer to \cite[Sec. 1]{peraire} for further details.
\end{rem}

\begin{rem}\label{hypo} Notice that, by (\ref{equal-coeficient}), $C_1$ and $C_2$ share the same maximal contact curves. In particular, given $H\in \mathbb{C}\{x\}[y]$ with $\deg_y(H)<\beta_0$ we get $v_1(H)=v_2(H)$. Moreover, if $H\in (x,y)$ then $v_1(\textup{d}H)=v_2(\textup{d}H)=v_1(H)=v_2(H)$ so $(\alpha,\alpha)\in\Lambda_{f_1f_2}$ for every $0\neq\alpha\in S_1=S_2$.
\end{rem}

Now we are ready to provide the generalization of Proposition \ref{prop:MBirreducible}.
    
\begin{lemma}\label{lem:expBM}Let $C_1$ and $C_2$ be two equisingular plane branches with semigroup $\langle \obeta_0,\ldots ,\obeta_g\rangle$ and intersection multiplicity $I\geq n_g\obeta_g$. If \(\omega = A \, dx - B \, dy \in \mathcal{E}(f_1) \cap f_1 \cdot \Omega^1(\log\ C_1)\) with \( B (f_1)_x + A (f_1)_y = M f_1\) where $v_1(B)=\sum_{l=0}^{q}\alpha_l\obeta_l$ and $\alpha_q\neq 0$, then  we have 
\[
\varphi_i^*(M) = \frac{e_k \obeta_{k+1}}{\obeta_0} b_i \, t_i^{v_1(B)-\obeta_0} + (\text{h.o.t.}),\quad
\varphi_i^*(B) = b_i \, t_i^{v_1(B)} + (\text{h.o.t.})\ \ \ \mbox{for}\ \ \ i=1,2
\]
where \(
k = \max_{0\leq l\leq g-1}\{e_l \nmid (v_1(B) - v_1((f_1)_y)) \}.
\) Moreover, \(b_1=b_2\) if 
$I>n_g\obeta_g$ or $q<g$. If \(I=n_g\obeta_g\) and \(q=g\) then \(b_i\) depends on the coefficients \(a_l^{(i)}\) with \(l\leq \beta_g.\)  
\end{lemma}

\begin{proof}
    From Proposition \ref{prop:MBirreducible}, we have the required expression for \(i=1\). 

Since $M, B\in\mathbb{C}\{x\}[y]$ with $\deg_y(M),\deg_y(B)<\beta_0$, by Remark \ref{hypo}, we get $v_2(B)=v_1(B)=\sum_{l=0}^q\alpha_l\obeta_l$ with $\alpha_q\neq 0$ and $v_2(M)=v_1(M)=v_1(B)-\obeta_0$. So, considering the $G$--expansion of $M$ and $B$, according to Remark \ref{qh}, for $i=1,2$ we have
\[
\varphi_i^*(M) =\varphi_i^*(b^M_{\alpha}\frakg_0^{\alpha_0-1}\frakg_1^{\alpha_1} \cdots \frakg_q^{\alpha_q})+(\text{h.o.t.})= Q^M_{v_1(M)}(a^{(i)}) t_i^{v_1(B)-\obeta_0} + (\text{h.o.t.}),\]
\[
\varphi_i^*(B) =\varphi_i^*(b^B_{\alpha}\frakg_0^{\alpha_0}\frakg_1^{\alpha_1} \cdots \frakg_q^{\alpha_q})+(\text{h.o.t.})=b_i t_i^{v_1(B)} + (\text{h.o.t.}).
\]
As $\varphi^*_i(\frakg_0)=t_i^{\obeta_0}$, independently of $i$, the coefficient of $t_i^{v_1(B)-\obeta_0}$ in $\varphi_i^*(\frakg_0^{\alpha_0-1}\frakg_1^{\alpha_1} \cdots \frakg_q^{\alpha_q})$ and the coefficient of $t_i^{v_1(B)}$ in $\varphi_i^*(\frakg_0^{\alpha_0}\frakg_1^{\alpha_1} \cdots \frakg_q^{\alpha_q})$ are the same, so \begin{equation}\label{relation-m-b}
\frac{b_i}{b^B_{\alpha}}=\frac{Q^M_{v_1(M)}(a^{(i)})}{b^M_{\alpha}},\ \ \ \mbox{that is}\ \ \ Q^M_{v_1(M)}(a^{(i)})=\frac{b^M_{\alpha}}{b^B_{\alpha}}b_i=\frac{e_k\obeta_{k+1}}{\obeta_0}b_i,
\end{equation}
where the last equality follows using $i=1$. 

By \cite[Lemma 1.7]{peraire}, $Q^M_{v_j(M)}(a^{(i)})$ and $b_i=Q^B_{v_j(B)}(a^{(i)})$ are non-zero homogeneous polynomials in the coefficients $a^{(i)}_k$ of $\varphi_i(t_i)$ with $k\leq\beta_q$. In this way, if \(q\neq g\) or \(I>n_g\obeta_g\), by (\ref{equal-coeficient}), we get
\[b_2=b_1\ \ \mbox{and}\ \ \ Q^M_{v_1(M)}(a^{(2)})=Q^M_{v_1(M)}(a^{(1)})=\frac{e_k \obeta_{k+1}}{\obeta_0} b_1,\]
that is, the coefficients are independent of the branch $C_i$. If \(q=g\) and \(I=n_g\obeta_g\) then $Q^M_{v_1(M)}(a^{(i)})$ and $Q^B_{v_1(B)}(a^{(i)})$ satisfy (\ref{relation-m-b}) and they depend on $a_{l}^{(i)}$ with $l\leq\beta_g$. 
\end{proof}

By Remark \ref{rem-inf-fiber}, given $\alpha\in\Lambda_f$ with $f=f_1f_2$, to characterize the infinite fiber \(F^{\Lambda}_{i}(\alpha)=\infty\) it is enough to characterize the values of differential forms in $f_j\cdot\Omega^1(log\ C_{j})$ with respect to $f_i$. Firstly, we will analyze the differential forms in $\mathcal{E}(f_j)\cap f_j\cdot\Omega^1(\log\ C_j)$. To do so, we will apply similar ideas to the ones in \cite[Sec. 4]{hernandessemiroots}.  

\begin{proposition}\label{prop:values+I}
Let $C_1$ and $C_2$ be two equisingular plane branches with semigroup $\langle \obeta_0,\ldots ,\obeta_g\rangle$ and intersection multiplicity $I$. Let \(\omega = A \, \textup{d}x - B \, \textup{d}y \in \mathcal{E}(f_1) \cap f_1 \cdot \Omega^1(\log\ C_1)\). If
\begin{enumerate}[a)]
     \item \(I>n_g\obeta_g\)\ \ or
     \item \(I=n_g\obeta_g\) and \(e_{g-1}\mid\;(\nu_1(B)-(c+\obeta_0-1))\)
 \end{enumerate}
 then
\begin{equation}\label{cond0}
\nu_2(\omega)=\nu_1(B)+I-(c+\obeta_0-1)=\nu_2(B)+I-(c+\obeta_0-1),
\end{equation}
where \(c\) denotes the conductor of \(\langle\obeta_0,\dots,\obeta_g\rangle.\)  In particular, $\omega\not\in f_2\cdot\Omega^1(\log\ C_2)$.
\end{proposition}
\begin{proof}
Since  \(\omega = A \, \textup{d}x - B \, \textup{d}y \in \mathcal{E}(f_1) \cap f_1 \cdot \Omega(\log\ C_1)\) there exists $M\in\mathbb{C}\{x\}[y]$ satisfying $B(f_1)_x+A(f_1)_y=Mf_1$ and, by (\ref{log-M}), \begin{equation}\label{fomega}
 (f_1)_y\omega=Mf_1\textup{d}x-B\textup{d}f_1.
\end{equation}

According to Proposition \ref{prop:MBirreducible} and Lemma \ref{lem:expBM}, we have

\[
\varphi_2^*(M)= \frac{e_k \obeta_{k+1}}{\obeta_0} b_2 \, t_2^{v_1(B)-\obeta_0} + (\text{h.o.t.}),\quad
\varphi_2^*(B) = b_2 \, t_2^{v_1(B)} + (\text{h.o.t.}),
\]
where \(
k = \max_{0\leq i\leq g-1}\{e_i \nmid (v_1(B) - v_1((f_1)_y)) \}\). In addition, we get
$$\varphi_2^*(f_1)=at_2^I+(\text{h.o.t.}).$$
Thus,
\begin{equation}\label{eqn:proofvalue+I}
    \begin{split}
        \varphi_2^*((f_1)_y\omega)&=\varphi_2^*(Mf_1\textup{d}x-B\textup{d}f_1)\\
        &=(ae_k\obeta_{k+1}b_2t_2^{v_1(B)+I-1}+\text{h.o.t.})(aIb_2t_2^{v_1(B)+I-1}+\text{h.o.t.})\\
        &=(e_k\obeta_{k+1}-I)ab_2t_2^{v_1(B)+I-1}+(\text{h.o.t.}).
    \end{split}
\end{equation}

If \(I>n_g\obeta_g=e_{g-1}\obeta_g\) then for all \(0\leq k\leq g-1\) we have \(I\neq e_{k}\obeta_{k+1}.\) If \(I=n_g\obeta_g\) and \(e_{g-1}\mid\;(\nu_1(B)-\nu_1((f_1)_y)\) then  \(k<g-1\) and thus \(I=n_g\obeta_g=e_{g-1}\obeta_g\neq e_k\obeta_{k+1}.\) So, for the condition $a)$ or $b)$ we get
$v_2((f_1)_y\omega)=v_1(B)+I$.

Since we have $B,(f_1)_y\in\mathbb{C}\{x\}[y]$ 
with $\deg_y(B),\deg_y((f_1)_y)<\beta_0$, it follows by Remark \ref{hypo} and (\ref{fy-cond}) that

$v_2(B)=v_1(B)$ and $v_2((f_1)_y)=v_1((f_1)_y)=c+\obeta_0-1$. So, \(v_2(\omega_2)=v_1(B)+I-(c-\obeta_0-1)=v_2(B)+I-(c-\obeta_0-1).\)
\end{proof}
\medskip

Let us now focus on the case of a reduced plane curve singularity with two equisingular branches and \(I=[f_1,f_2]_0>n_g\obeta_g.\) Recall that given $\omega=A\textup{d}x-B\textup{d}y\in \Omega^1$, its {\it Weierstrass 1-form} (see \cite{GHI23}) of $\omega$ with respect to $f_1$ is given as 
\begin{equation}\label{exp}
	\omega=\omega_0+\omega_1
\end{equation}
with $\omega_0=Pf_1\textup{d}x-Q\textup{d}f_1\in \mathcal{G}(f_1)$ and $\omega_1=A_1\textup{d}x-B_1\textup{d}y\in\mathcal{E}(f_1)$.

\begin{rem}\label{remark1}
	For \(i,j\in\{1,2\},\) \(i\neq j,\) given $\omega\in\mathcal{G}(f_j)=\mathbb{C}\{x,y\}f_j\textup{d}x+\mathbb{C}\{x,y\}\textup{d}f_j\subset f_j\cdot\Omega^1(\log\ C_j)$ we get 
	$v_i(\omega)\geq I=[f_i,f_j]_0=v_i(f_j)=v_i(\textup{d}f_j)$. In addition we have that    
	$$I+S_i\subset v_i(\mathcal{G}(f_j))\subset v_i(f_j\cdot\Omega^1(\log\ C_j)).$$
	In fact, given $\gamma\in S_i$ we consider $h\in\mathbb{C}\{x,y\}$ such that $v_i(h)=\gamma$. Since $h\textup{d}f_j\in\mathcal{G}(f_j)\subset f_j\cdot\Omega^1(\log\ C_j)$, we get $v_i(\textup{h}df_j)=\gamma+I\in v_i(\mathcal{G}(f_j))\subset v_i(f_j\cdot\Omega^1(\log\ C_j))$.
\end{rem}

Notice that Remark \ref{remark1} implies that in order to compute the infinite fibers for which we need to work a bit we only need to look at those \(\omega\in f_j\cdot\Omega^1(\log\ C_j)\) for which \(v_i(\omega)\leq I+c.\)

\begin{rem}\label{remark2}
Given $\omega=A\textup{d}x-B\textup{d}y\in f_1\cdot\Omega^1(\log\ C_1)$ we consider its Weierstrass 1-form with respect to $f_1$, that is $\omega=\omega_0+\omega_1$ with $\omega_0\in\mathcal{G}(f_1)$ and $\omega_1\in\mathcal{E}(f_1)$. Since
$res(\omega)=res(\omega_0)+res(\omega_1),$ we have $v_1(res(\omega))\geq\min\{v_1(res(\omega_0)),v_1(res(\omega_1))\}$. In this way if $v_1(res(\omega))<0$ then, since $v_1(res(\omega_0))\geq 0$ by construction, we get $v_1(res(\omega))=v_1(res(\omega_1))$.
\end{rem}

\begin{lemma}\label{lemma1}
Suppose that $I>n_gv_g$. Given $\omega=\omega_0+\omega_1\in f_1\cdot\Omega^1(\log\ C_1)$ as before. 

If $v_1(res(\omega))<0$ then
$v_2(\omega)=v_2(\omega_1)=I+v_2(res(\omega))=I+v_2(res(\omega_1))=I+v_1(res(\omega_1)).$	
\end{lemma}
\begin{proof}
Given $\omega=A\textup{d}x-B\textup{d}y\in f_1\cdot\Omega^1(\log C_1)$ we consider its Weierstrass 1-form with respect to $f_1$, that is $\omega=\omega_0+\omega_1$ with $\omega_0\in\mathcal{G}(f_1)$ and $\omega_1:=A_1\textup{d}x-B_1\textup{d}y\in\mathcal{E}(f_1)$. By Remark \ref{remark2} we get $v_1(res(\omega))=v_1(res(\omega_1))$ where $res(\omega_1)=B_1/(f_1)_y$.

By hypothesis $I>n_gv_g,$ then since $deg_y(B_1), deg_y((f_1)_y)<\obeta_0$, by Remark \ref{hypo}, it follows that $$v_2(res(\omega_1))=v_2(B_1)-v_2((f_1)_y)=v_1(B_1)-v_1((f_1)_y)=v_1(res(\omega_1))$$ and by Proposition \ref{prop:values+I}, we get 
$$v_2(\omega_1)=I+v_2(res(\omega_1))=I+v_1(res(\omega))<I.$$

On the other hand, $v_2(\omega)\geq\min\{v_2(\omega_0),v_2(\omega_1)\}$. By Remark \ref{remark1}, $v_2(\omega_0)\geq I$ and as $v_2(\omega_1)=I+v_2(res(\omega_1))<I$ it follows that

$$v_2(\omega)=v_2(\omega_1)=I+v_2(res(\omega))=I+v_2(res(\omega_1))=I+v_1(res(\omega_1)).$$
\end{proof}

In addition, we have the following

\begin{lemma}\label{lemma2}
	Suppose that $I>n_gv_g$, then $\con (v_2(f_1\cdot\Omega^1(\log\ C_1)))\leq I-\obeta_0+1$.
\end{lemma}
\begin{proof}
It is sufficient to show that for any $k\in\mathbb{N}$ with $0\leq k\leq \obeta_0-1$ we have 
\[I-k+\obeta_0\mathbb{N}\subset v_2(f_1\cdot\Omega^1(\log\ C_1)).\]

For $k=0$ if we take $n\in\mathbb{N}$ then $x^n\textup{d}f_1\in \mathcal{G}(f_1)\subset f_1\cdot\Omega^1 (\log\ C_1)$ and 
\[I+n\obeta_0=v_2(x^n\textup{d}f_1)\in v_2(f_1\cdot\Omega^1 (\log\ C_1)).\]

Now consider $k\in\mathbb{N}$ such that $0<k\leq\obeta_0-1$. By \cite[Prop. 3.21]{PolLogarithmic}, we have that $\con(\Delta_{f_1})=-\obeta_0+1$. Thus, we get $-k\in v_1(res(f_1\cdot\Omega^1(\log\ C_1)))$ so, there exists $\omega\in f_1\cdot\Omega^1(\log\ C_1)$ such that $v_1(res(\omega))=-k<0$. In this way, by Lemma \ref{lemma1}, $v_2(\omega)=I+v_1(res(\omega))=I-k$. Taking any $n\in\mathbb{N}$ we get
$$I-k+n\obeta_0=v_2(x^n\omega)\in v_2(f_1\cdot\Omega^1(\log\ C_1)).$$

Hence 
$$I-\obeta_0+1+\mathbb{N}=\bigcup_{k=0}^{\obeta_0-1}(I-k+\obeta_0\mathbb{N})\subseteq v_2(f_1\cdot\Omega^1(\log\ C_1))$$ and consequently, $\con (v_2(f_1\cdot\Omega^1(\log\ C_1)))\leq I-\obeta_0+1.$
\end{proof}

Now we are able to describe the infinite fibers of $\Lambda_f$ where $f=f_1f_2$ and $I>n_gv_g$.

\begin{theorem}\label{thm:fibers+I}
	Suppose that $I>n_gv_g$, then    
	$$v_2(f_1\cdot\Omega^1(\log\ C_1))=I+\Delta_{f_1}\ \ \mbox{and}\ \ \con(v_2(f_1\cdot\Omega^1(\log\ C_1)))=I-\obeta_0+1.$$
\end{theorem}
\begin{proof}
Let $\omega\in f_1\cdot\Omega^1(\log\ C_1)$ and $\omega=\omega_0+\omega_1$ its Weierstrass 1-form with respect to $f_1$, that is $\omega_0\in \mathcal{G}(f_1)\cap f_1\cdot\Omega^1(\log\ C_1)$ and $\omega_1=A_1\textup{d}x-B_1\textup{d}y\in\mathcal{E}(f_1)\cap f_1\cdot\Omega^1(\log\ C_1)$.
\medskip

We have that $v_2(\omega)\geq\min\{v_2(\omega_0), v_2(\omega_1)\}$.
\medskip

If $v_2(\omega)\geq v_2(\omega_0)$, then by Remark \ref{remark1} we get $v_2(\omega)\geq I$ and $v_2(\omega)\in I+\Delta_{f_1}$ because $\con(\Delta_{f_1})=-\obeta_0+1$. If $v_2(\omega)=v_2(\omega_1)$, then by Proposition \ref{prop:values+I} we get $v_2(\omega_1)=I+v_1(B_1)-v_1((f_1)_y)\in I+\Delta_{f_1}$. Hence, $v_2(f_1\cdot\Omega^1(\log\ C_1))\subseteq I+\Delta_{f_1}$.
\medskip

On the other hand, let us consider $I+\delta\in I+\Delta_{f_1}$, that is, $\delta\in\Delta_{f_1}$.
\medskip

If $\delta\geq -\obeta_0+1$ then, by Lemma \ref{lemma2}, there exists $\omega\in f_1\cdot\Omega^1(\log\ C_1)$ such that $v_2(\omega)=I+\delta$.

If $-\obeta_0+1>\delta\in\Delta_{f_1}$ then there exists $\omega\in f_1\cdot\Omega^1(\log\ C_1)$ such that $\delta=v_1(res(\omega))<0$. So, by Lemma \ref{lemma1}, 
we get
$$v_2(\omega)=v_2(\omega_1)=I+v_2(res(\omega_1))=I+v_1(res(\omega_1))=I+\delta.$$
In this way, $I+\Delta_{f_1}\subseteq v_2(f_1\cdot\Omega^1(\log\ C_1))$. This conclude that $v_2(f_1\cdot\Omega^1(\log\ C_1))= I+\Delta_{f_1}$ and, since $\con(\Delta_{f_1})=-\obeta_0+1$, it follows that $\con(v_2(f_1\cdot\Omega^1(\log\ C_1)))=I-\obeta_0+1$.
\end{proof}

As a consequence we obtain
\begin{theorem}\label{thm:conductorkahler1}
     Let \(f=f_1f_2\) such that \(C_1,C_2\) are equisingular with values semigroup $\langle\obeta_0,\ldots ,\obeta_g\rangle$, \(I> n_g\obeta_g\) and \(\Lambda\) its value set of K\"ahler differentials. Then,
\begin{equation}\label{conductorkahler1}
    \con(\Lambda)=
       (I-\obeta_0+1,I-\obeta_0+1)  
    \end{equation}
    In particular, \(
    \con(\Lambda)\) is independent of the analytic type of each of the branches.
\end{theorem}
\begin{proof}
    Since $I>n_g\obeta_g$ it follows by (\ref{eq:definbetabarra}) that $I-\obeta_0>\con(S_1)$, by (\ref{fy-cond}) and Remark \ref{hypo} we get $(\delta,\delta)\in\Lambda$ for any $\delta\geq I-\obeta_0$. In addition, we have that \(\obeta_0=\min\Lambda_1=\min\Lambda_2,\) and Theorem \ref{thm:fibers+I} implies that \(\con_{\Lambda}\leq (I-\obeta_0+1,I-\obeta_0+1).\) To prove the equality, it is enough to show that \((I-\obeta_0,I-\obeta_0)\in\Lambda\) is a maximal point (in fact an absolute maximal) of \(\Lambda.\)
\medskip

    By Theorem \ref{thm:fibers+I}, we get \(F_{\{i\}}^{\Lambda}(I-\obeta_0,I-\obeta_0)\neq \infty\) for \(i\in\{1,2\}\) (see Definition \ref{inf-fiber}). Let us consider for example \(i=2,\) as the other case follows similarly. As \(F_{\{2\}}^{\Lambda}(I-\obeta_0,I-\obeta_0)\neq\infty\) then there exist \(\alpha \in \overline{F}_{\{2\}}(\Lambda,(I-\obeta_0,I-\obeta_0)) \) such that \(\alpha\) is a maximal element of \(\Lambda.\) As the number of branches is \(r=2\) then the notion of maximal, relative maximal and absolute maximal agree. Hence \(\alpha=(\alpha_1,I-\obeta_0)\) with $\alpha_1\geq I-\obeta_0$ is an absolute maximal in \(\Lambda.\) Let us assume \(\alpha\neq (I-\obeta_0,I-\obeta_0),\) then \(\alpha=(I-\obeta_0+n,I-\obeta_0)\) for some \(n\in\mathbb{N}\setminus\{0\}\) is such that \(F_{\{1\}}(\Lambda,\alpha)=F_{\{2\}}(\Lambda,\alpha)=\emptyset\) and \(\alpha_1\in I+\Delta_2\) since $\con(\Delta_2)=-\obeta_0+1,$ but this is a contradiction with Theorem \ref{thm:fibers+I}.
\end{proof}

The previous results allows us to compute the invariant $\Theta_2$ described in Section 2 for the fractional ideal $J=\frac{\Omega_f}{Tor(\Omega_f)}\simeq \varphi^*(\Omega_f)$ (see \ref{quot-torsion}) whose values set is $E=\Lambda_f$. In fact, according to Remark \ref{rem-inf-fiber} we get
$\omega\in f_1\cdot \Omega^1(\log\ C_1)$ if and only if $v_1(\overline{\omega})=\infty$. So, we have that $\mathcal{N}_1(J)=f_1\cdot \Omega^1(\log\ C_1)$ and, by (\ref{thetai}) we get
$\Theta_2=\sharp (\Lambda_2\setminus v_2(f_1\cdot \Omega^1(\log\ C_1)))$.

\begin{corollary}\label{cor-theta}
With the previous notation we have that $$\Theta_2=I-\beta_0+1-\sharp\mathbb{N}\setminus \Lambda_2-\sharp\{\lambda>\beta_0 :\ \lambda\not\in\Lambda_1\}.$$
\end{corollary}
\begin{proof}
By Theorem \ref{thm:fibers+I} we have that $v_2(f_1\cdot\Omega^1(\log\ C_1))=I+\Delta_{f_1}$ and $\con(v_2(f_1\cdot\Omega^1(\log\ C_1)))=I-\beta_0+1$.
Since $\con (\Lambda_2)\leq \con(v_2(f_1\cdot\Omega^1(\log\ C_1)))$ we have
$$\begin{array}{ll}\Theta_2 & =\sharp (\Lambda_2\setminus v_2(f_1\cdot \Omega^1(\log\ C_1))) \\
& =\sharp \{\alpha\in\Lambda_2:\ \alpha<I-\beta_0+1\}-\sharp\{\delta\in I+\Delta_{f_1}:\ \delta<I-\beta_0+1\} \\
& =I-\beta_0+1-\sharp \mathbb{N}\setminus\Lambda_2-\sharp\{\gamma\in\Delta_{f_1}: \gamma<-\beta_0+1\}
\end{array}.$$
It follows, by (\ref{relation}), that
\(\Theta_2=I-\beta_0+1-\sharp\mathbb{N}\setminus \Lambda_2-\sharp\{\lambda>\beta_0 :\ \lambda\not\in\Lambda_1\}\).
\end{proof}



\section{The Tjurina number for two branches}\label{sec:Tjurina2}

In Section \ref{sec:logdif} we have computed the conductor \(\con(\Lambda)\) of the value set \(\Lambda\) of K\"ahler differentials for a plane curve $C=C_1\cup C_2$ defined by \(f=f_1f_2\) such that \(C_1,C_2\) are equisingular with intersection multiplicity \(I> n_g\obeta_g\). Also, in Theorem \ref{thm:Tjurinaformula} we have shown that the Tjurina number of a reduced plane curve with any number of branches can be computed in terms of the distance \(d(\overline{\Lambda}\setminus S)\) between the semigroup of values of the curve \(S\) and the set \(\overline{\Lambda}=\Lambda\cup\{0\}.\) In addition, in section \ref{sec:semigroup} we explained how to compute this distance if one knows the conductor of the module. All this together allows us to provide explicit formulas for the Tjurina number in this case. 

\begin{theorem}\label{thm:taumin1}
   Let $C=C_1\cup C_2$ be a plane curve defined such that \(C_1,C_2\) are equisingular with values semigroup $S_i=\langle \obeta_0,\ldots ,\obeta_g\rangle$, conductor $\con(S_i)=c$ and intersection multiplicity \(I> n_g\obeta_g\).  Then, the Tjurina number of $C$ is given by
\[
  \tau=2I+c-1.\]
    In particular, \(\tau\) is constant in the equisingularity class of $C$. 
\end{theorem}
\begin{proof}
By Theorem \ref{thm:Tjurinaformula} and Remark \ref{cond-milnor} we get
\begin{equation}\label{tau1}
\tau=\mu-d(\overline{\Lambda}\setminus S)=2c+2I-1-d(\overline{\Lambda}\setminus S).
\end{equation}
Since the values semigroup $S$ of $C$ is such that $S\subseteq\overline{\Lambda}=\Lambda\cup\{(0,0)\}$ and $m_{\Lambda}=(\beta_0,\beta_0)$ we get 
\begin{equation}\label{eqn:auxpruebatau0}
d(\overline{\Lambda}\setminus S)  =d_{\overline{\Lambda}}((0,0),\con(S))-d_S((0,0),\con(S))  =d_\Lambda((\beta_0,\beta_0),\con(S))+1-(c+I),\end{equation}
where $d_S((0,0),\con(S))=\delta(C)=c+I$ (see Example \ref{delta}).

By Theorem \ref{thm:conductorkahler1} we have \(\con(\Lambda)=(I-\beta_0+1,I-\beta_0+1)<(I+c,I+c)=\con(S)\), so we get \begin{equation}\label{eqn:auxpruebatau1}d_\Lambda((\beta_0,\beta_0),\con(S))=d_\Lambda((\beta_0,\beta_0),\con(\Lambda))+d_\Lambda(\con(\Lambda),\con(S)).\end{equation}

If $J=\varphi^*(\Omega_f)$ then, by  Theorem \ref{thm:colengthcalculation}, we have that 
\[d_\Lambda((\beta_0,\beta_0),\con(\Lambda))=l\left ( \frac{J}{J(\con(\Lambda))}\right ) =\sum_{i=1}^{2}(I-\beta_0+1-\beta_0-\sharp ((\mathbb{N}+\beta_0)\setminus \Lambda_i)-\Theta_i),\]
where $\Theta_1=0$ and, by Corollary \ref{cor-theta}, $\Theta_2=I-\beta_0+1-\sharp\mathbb{N}\setminus \Lambda_2-\sharp\{\lambda>\beta_0 :\ \lambda\not\in\Lambda_1\}$.

Since, $\sharp\mathbb{N}\setminus\Lambda_2=\beta_0+\sharp((\mathbb{N}+\beta_0)\setminus\Lambda_2)$ and $(\mathbb{N}+\beta_0)\setminus\Lambda_1=\{\lambda>\beta_0:\ \lambda\not\in\Lambda_1\}$ we get 
\begin{equation}
    \label{eqn:auxpruebatau2}
    d_\Lambda((\beta_0,\beta_0),\con(\Lambda))=I-2\beta_0+1.
\end{equation}

Considering $J=\{\omega\in\varphi^*(\Omega_f): \vu (\omega)\geq\con(\Lambda)\}$ then $\vu(J)=(I-\beta_0+1,I-\beta_1+1)+\mathbb{N}^2$ and, by Theorem \ref{thm:colengthcalculation}, we have that \begin{equation}    \label{eqn:auxpruebatau3}   d_\Lambda(\con(\Lambda),\con(S))=l\left ( \frac{J}{J(\con(S))}\right )=2(I+c-(I-\beta_0+1))=2(c+\beta_0-1).\end{equation}

The expressions (\ref{eqn:auxpruebatau2}) and (\ref{eqn:auxpruebatau3}) give us, by (\ref{eqn:auxpruebatau1}),  that $d_\Lambda((\beta_0,\beta_0),\con(S))=I+2c-1$. So, by (\ref{eqn:auxpruebatau0}) we have $d(\overline{\Lambda}\setminus S)  =c$ and, consequently by (\ref{tau1})
$$\tau=2I+c-1.$$
\end{proof}

Notice that, by the previous result, we get $\mu-\tau=c=\mu_1=\mu_2$ for a plane curve $C=C_1\cup C_2$ with $C_1$ and $C_2$ equisingular and intersection multiplicity $I>n_g\obeta_g$ where $\mu_i=c$ is the Milnor number of $C_i$.

\begin{ex}
Let us consider
\[\begin{array}{l}
f=y^6-3x^3y^4-2x^5y^3+3x^6y^2-6x^8y-x^9+x^{10}\vspace{0.2cm}\\
g=y^6-3x^3y^4+4x^5y^3+\left ( 3x^6-\frac{3}{2}x^7\right )y^2-\left (12x^8+\frac{3}{8}x^9 \right )y-x^9+\frac{13}{2}x^{10}-\frac{1}{64}x^{11}.\vspace{0.2cm}\\
h=y^6-3y^4x^3+\left(-\frac{6}{19}x^6+4x^5\right )y^3+\left (-\frac{3}{27436}x^{10}+\frac{9}{76}x^8-\frac{3}{2}x^7+3x^6\right )y^2+\\ \hspace{0.75cm}+\left (-\frac{3}{1042568}x^{13}+\frac{3}{13718}x^{12}-\frac{9}{1444}x^{11}+\frac{3}{38}x^{10}-\frac{9}{152}x^9-12x^8\right )y-\frac{1}{3010936384}x^{17}\vspace{0.2cm}\\ \hspace{0.75cm}+\frac{3}{79235168}x^{16}-\frac{15}{8340544}x^{15}+\frac{5}{109744}x^{14}-\frac{301}{438976}x^{13}+\frac{153}{11552}x^{12}-\frac{547}{1216}x^{11}+\frac{13}{2}x^{10}-x^9\end{array}\]    
We have that $f, g$ and $h$ are equisingular plane branches sharing the semigroup $S=\langle 6,9,19\rangle$ with conductor $c=42$ and  $I(g,h)=63>I(f,h)=I(f,g)=58>n_2\overline{\beta}_2=57$. 

In this case, using the \textsc{Singular}    
software \cite{singular}, we get $\tau(f)=35, \tau(g)=36, \tau(h)=37$, 
\[\tau(fg)=2I(f,g)+c-1=157=2I(f,h)+c-1=\tau(f,h)\ \ \mbox{and}\ \ \tau(gh)=169=2I(g,h)+c-1.\]
that illustrate Theorem \ref{thm:taumin1}.
\end{ex}

In \cite{quotpat}, the inequality \(\mu/\tau<4/3\) was showed for any plane curve. The previous results allows to provide a new proof of that inequality in the case of a plane curve with two equisingular branches and \(I>n_g\obeta_g.\)
\begin{corollary}
   Let $C=C_1\cup C_2$ be a plane curve defined such that \(C_1,C_2\) are equisingular and intersection multiplicity \(I> n_g\obeta_g\).  Then, \(\mu/\tau<4/3.\)
\end{corollary}
\begin{proof}
    From Theorem \ref{thm:taumin1} we have 
    \(4\tau-3\mu=2I-2c-1\) where \(c=\mu_i\) is the Milnor number of the branch $C_i$. By hypothesis \(I>n_g\obeta_g>c\), then the results follows.
\end{proof}

\subsection{Remarks on the minimal Tjurina number in more general cases}
To conclude, we draw attention to some challenges concerning the Tjurina number in a more general setting.
\medskip

First, Theorem \ref{thm:taumin1} establishes that for a curve with two equisingular branches, the condition \(I>n_{g}\obeta_g\) is sufficient to guarantee a constant Tjurina number within the equisingularity class. It is natural to ask up to what extent this holds in the case of a curve with more than two branches. The following example shows that if \(C=\cup_{i=1}^{r}C_i\) is a curve with \(r\geq 3\) branches, then condition \(I_{i,j}:=[C_i,C_j]_0>n_{g}\obeta_g\) is not enough to have constant Tjurina number in the whole equisingularity class.
\begin{ex}
    Consider \(f_1=y^5-x^8+2x^5y^2,\) \(f_2=y^5-x^8+3x^5y^2,\) \(f_3=y^5-x^8+x^4y^3\) and \(f_4=y^5-x^8+7x^5y^2.\) We have that \(f=f_1 f_2 f_3\) and \(g=f_1 f_2 f_4\) define two equisingular curves each one with three
     equisingular branches with values semigroup $\langle 5,8\rangle$, all of them have intersection multiplicity \(I_{i,j}=41>40=n_{g}\obeta_g.\) A computation with {\sc Singular} \cite{singular} shows \(\tau(f)=258\neq261=\tau(g).\)
\end{ex}
It would be certainly good to obtain sufficient topological conditions for a curve with several branches to have constant Tjurina number in the equisingularity class. In \cite{hernandessemiroots,LuengoPfister88} some families of plane branches with constant Tjurina number in the equisingularity class are shown.
One could think that for a curve with two branches, to have constant Tjurina number in the equisingularity class of each of the branches could be a sufficient condition to have constant Tjurina number in the equisingularity class of the curve. The following example shows that this is not enough.

\begin{ex}\label{ex:conjeture}Let us consider the branches\[ \begin{array}{l}f_1=(y^5-x^7)^2-x^{10}y^3,\ \ \ f_2=(y^5-x^7)^2-5x^{10}y^3\ \ \ \text{and}\ \ \ f_3=(y^5-x^7+x^4y^3)^2-3x^{10}y^3.\end{array}\]All branches are equisingular with semigroup \(\langle 10,14,71\rangle.\) and, by \cite{hernandessemiroots}, for any branch in this equisingularity class the Tjurina number is constant \(\tau(f_i)=94.\) Let us denote \(f=f_1f_2\) and \(g=f_2f_3.\) In both cases \([f_i,f_j]_0=142=2\cdot 71=n_{g}\obeta_g.\) A calculation with {\sc Singular} \cite{singular} shows that \(\tau(f)=402\neq 406=\tau(g).\)
\end{ex}

Following with Example \ref{ex:conjeture}, we observe that it is quite close to the curves considered in Theorem \ref{thm:taumin1}. The difficulty here relies on the remaining cases of (b) of Proposition \ref{prop:values+I}, i.e. to compute values, with respect to $f_2$, of those differentials \(\omega\in \mathcal{E}(f_1)\cap f_1\cdot\Omega^1(\log\ C_1)\) such that \(e_{g-1}\nmid res(\omega).\) In that cases, one can check that the initial term in Equation \eqref{eqn:proofvalue+I} cancels and one need to impose some open conditions in order to guarantee the value of $\omega$. A careful analysis of this situation leads us to think that in the case of a plane curve with two equisingular branches with \(I=n_{g}\obeta_g\) there should exists an open Zariski set for which \(v_2(f_1\cdot\Omega^1(\log\ C_1))=I+1+\Delta_{f_1}.\) This leads us to propose the following conjecture.
\begin{conj}\label{conj:taumin2}
    Let \(C=C_1\cup C_2\) be a plane curve with two equisingular branches with semigroup \(S=S_1=S_2=\langle\obeta_0,\dots,\obeta_g\rangle\) and \(I=[C_1,C_2]_0=n_{g}\obeta_g.\) Denote by \(c\) the conductor of \(S.\) Then, the minimal Tjurina number in the equisingularity class of $C$ is
    \[\tau_{min}=2I+c.\]
\end{conj}
In fact, Conjecture \ref{conj:taumin2} is actually true for \(S=\langle\obeta_0,\obeta_1\rangle\) as showed in \cite[Tableu 3, \(\delta\) pair]{BGM88} (see also \cite{heroy89}).

\medskip

\printbibliography

@article {peraire,
    AUTHOR = {Peraire, R.},
     TITLE = {Tjurina number of a generic irreducible curve singularity},
   JOURNAL = {J. Algebra},
  FJOURNAL = {Journal of Algebra},
    VOLUME = {196},
      YEAR = {1997},
    NUMBER = {1},
     PAGES = {114--157},
      ISSN = {0021-8693,1090-266X},
   MRCLASS = {32S15 (14B12 14H20 32S30)},
  MRNUMBER = {1474166},
MRREVIEWER = {Gerhard\ Pfister},
       DOI = {10.1006/jabr.1997.7073},
       URL = {https://doi.org/10.1006/jabr.1997.7073},
}

@article {Delgmanuscripta1,
    AUTHOR = {Delgado de la Mata, F.},
     TITLE = {The semigroup of values of a curve singularity with several
              branches},
   JOURNAL = {Manuscripta Math.},
  FJOURNAL = {Manuscripta Mathematica},
    VOLUME = {59},
      YEAR = {1987},
    NUMBER = {3},
     PAGES = {347--374},
      ISSN = {0025-2611,1432-1785},
   MRCLASS = {14H20},
  MRNUMBER = {909850},
MRREVIEWER = {Joan\ Elias},
       DOI = {10.1007/BF01174799},
       URL = {https://doi.org/10.1007/BF01174799},
}

@incollection {Delgadoari,
    AUTHOR = {Delgado de la Mata, F.},
     TITLE = {An arithmetical factorization for the critical point set of some map germs from {${\bf C}^2$} to {${\bf C}^2$}},
 BOOKTITLE = {Singularities ({L}ille, 1991)},
    SERIES = {London Math. Soc. Lecture Note Ser.},
    VOLUME = {201},
     PAGES = {61--100},
 PUBLISHER = {Cambridge Univ. Press, Cambridge},
      YEAR = {1994},
      ISBN = {0-521-46631-8},
   MRCLASS = {32S05 (14H20)},
  MRNUMBER = {1295072},
MRREVIEWER = {Arkadiusz\ P\l oski},
}

@article{hernandesSemiring,
    AUTHOR = {de Carvalho, E. and Hernandes, M. E.},
     TITLE = {The value semiring of an algebroid curve},
   JOURNAL = {Comm. Algebra},
  FJOURNAL = {Communications in Algebra},
    VOLUME = {48},
      YEAR = {2020},
    NUMBER = {8},
     PAGES = {3275--3284},
      ISSN = {0092-7872,1532-4125},
   MRCLASS = {14H20 (14H50 14Q05 16Y60)},
  MRNUMBER = {4115349},

       DOI = {10.1080/00927872.2020.1733588},
       URL ={https://doi.org/10.1080/00927872.2020.1733588},
}

@book {zariski,
    AUTHOR = {Zariski, O.},
     TITLE = {Le probl\`eme des modules pour les branches planes},
   EDITION = {Second},
      NOTE = {Course given at the Centre de Math\'{e}matiques de l'\'{E}cole
              Polytechnique, Paris, October--November 1973,
              With an appendix by Bernard Teissier},
 PUBLISHER = {Hermann, Paris},
      YEAR = {1986},
     PAGES = {x+212},
      ISBN = {2-7056-6036-4},
   MRCLASS = {14H20},
  MRNUMBER = {861277},
MRREVIEWER = {Jonathan M. Wahl},
}

@article {CDGlondon,
    AUTHOR = {Campillo, A. and Delgado de la Mata, F. and Gusein-Zade, S. M.},
     TITLE = {On generators of the semigroup of a plane curve singularity},
   JOURNAL = {J. London Math. Soc. (2)},
  FJOURNAL = {Journal of the London Mathematical Society. Second Series},
    VOLUME = {60},
      YEAR = {1999},
    NUMBER = {2},
     PAGES = {420--430},
      ISSN = {0024-6107},
   MRCLASS = {14B05},
  MRNUMBER = {1724869},
MRREVIEWER = {Mutsuo Oka},
       DOI = {10.1112/S0024610799007917},
       URL = {https://doi.org/10.1112/S0024610799007917},
}

@article {hernandessemiroots,
    AUTHOR = {de Abreu, M. O. R. and Hernandes, M. E.},
     TITLE = {On the analytic invariants and semiroots of plane branches},
   JOURNAL = {J. Algebra},
  FJOURNAL = {Journal of Algebra},
    VOLUME = {598},
      YEAR = {2022},
     PAGES = {284--307},
      ISSN = {0021-8693,1090-266X},
   MRCLASS = {14H20 (32S10)},
  MRNUMBER = {4379283},
MRREVIEWER = {Nuria\ Corral},
       DOI = {10.1016/j.jalgebra.2022.01.032},
       URL = {https://doi.org/10.1016/j.jalgebra.2022.01.032},
}

@article {PolLogarithmic,
    AUTHOR = {Pol, D.},
     TITLE = {On the values of logarithmic residues along curves},
   JOURNAL = {Ann. Inst. Fourier (Grenoble)},
  FJOURNAL = {Universit\'e{} de Grenoble. Annales de l'Institut Fourier},
    VOLUME = {68},
      YEAR = {2018},
    NUMBER = {2},
     PAGES = {725--766},
      ISSN = {0373-0956,1777-5310},
   MRCLASS = {14H20 (14B07 32A27)},
  MRNUMBER = {3803117},
MRREVIEWER = {P.\ Schenzel},
       DOI = {10.5802/aif.3176},
       URL = {https://doi.org/10.5802/aif.3176},
}

@article {Danna1,
    AUTHOR = {D'Anna, M.},
     TITLE = {The canonical module of a one-dimensional reduced local ring},
   JOURNAL = {Comm. Algebra},
  FJOURNAL = {Communications in Algebra},
    VOLUME = {25},
      YEAR = {1997},
    NUMBER = {9},
     PAGES = {2939--2965},
      ISSN = {0092-7872,1532-4125},
   MRCLASS = {13H10 (13J10 14B05)},
  MRNUMBER = {1458740},
MRREVIEWER = {David\ E.\ Dobbs},
       DOI = {10.1080/00927879708826033},
       URL = {https://doi.org/10.1080/00927879708826033},
}

@article {Danna2,
    AUTHOR = {Barucci, V. and D'Anna, M. and Fr\"oberg, R.},
     TITLE = {Analytically unramified one-dimensional semilocal rings and
              their value semigroups},
   JOURNAL = {J. Pure Appl. Algebra},
  FJOURNAL = {Journal of Pure and Applied Algebra},
    VOLUME = {147},
      YEAR = {2000},
    NUMBER = {3},
     PAGES = {215--254},
      ISSN = {0022-4049,1873-1376},
   MRCLASS = {13H99 (13A18)},
  MRNUMBER = {1747441},
MRREVIEWER = {Ngo Viet Trung},
       DOI = {10.1016/S0022-4049(98)00160-1},
       URL = {https://doi.org/10.1016/S0022-4049(98)00160-1},
}

@article {BGM88,
    AUTHOR = {Brian\c{c}on, J. and Granger, M. and Maisonobe, Ph.},
     TITLE = {Le nombre de modules du germe de courbe plane {$x^a+y^b=0$}},
   JOURNAL = {Math. Ann.},
  FJOURNAL = {Mathematische Annalen},
    VOLUME = {279},
      YEAR = {1988},
    NUMBER = {3},
     PAGES = {535--551},
      ISSN = {0025-5831,1432-1807},
   MRCLASS = {14B05 (32B30)},
  MRNUMBER = {922433},
MRREVIEWER = {Gerhard\ Pfister},
       DOI = {10.1007/BF01456286},
       URL = {https://doi.org/10.1007/BF01456286},
}

@article {Berger63,
	AUTHOR = {Berger, R.},
	TITLE = {Differentialmoduln eindimensionaler lokaler {R}inge},
	JOURNAL = {Math. Z.},
	FJOURNAL = {Mathematische Zeitschrift},
	VOLUME = {81},
	YEAR = {1963},
	PAGES = {326--354},
	ISSN = {0025-5874,1432-1823},
	MRCLASS = {13.60 (13.95)},
	MRNUMBER = {152546},
	MRREVIEWER = {P.\ Samuel},
	DOI = {10.1007/BF01111579},
	URL = {https://doi.org/10.1007/BF01111579},
}

@article {Hefezcolength,
    AUTHOR = {de Guzm\'an, E. M. N. and Hefez, A.},
     TITLE = {On the colength of fractional ideals},
   JOURNAL = {J. Singul.},
  FJOURNAL = {Journal of Singularities},
    VOLUME = {21},
      YEAR = {2020},
     PAGES = {119--131},
      ISSN = {1949-2006},
   MRCLASS = {13H10 (14H20)},
  MRNUMBER = {4084203},
MRREVIEWER = {Nil\ \c Sahin},
       DOI = {10.5427/jsing.2020.21g},
       URL = {https://doi.org/10.5427/jsing.2020.21g},
}

@article{HernandesRodrigues, 
title={The analytic classification of plane curves}, 
volume={160}, 
DOI={10.1112/S0010437X24007061}, 
number={4}, 
journal={Compositio Mathematica}, 
author={Hernandes, M. E. and Rodriges Hernandes, M. E.}, 
year={2024}, 
pages={915–944}}

@article {delgadogorenstein,
    AUTHOR = {Delgado de la Mata, F.},
     TITLE = {Gorenstein curves and symmetry of the semigroup of values},
   JOURNAL = {Manuscripta Math.},
  FJOURNAL = {Manuscripta Mathematica},
    VOLUME = {61},
      YEAR = {1988},
    NUMBER = {3},
     PAGES = {285--296},
      ISSN = {0025-2611,1432-1785},
   MRCLASS = {14H20 (13H10 13J10)},
  MRNUMBER = {949819},
MRREVIEWER = {J\"urgen\ Herzog},
       DOI = {10.1007/BF01258440},
       URL = {https://doi.org/10.1007/BF01258440},
}

@misc{hefezhernandes24,
      title={Colengths of fractional ideals and Tjurina number of a reducible plane curve}, 
      author={Hefez, A. and Hernandes, M. E.},
      year={2024},
      eprint={2409.11153},
      archivePrefix={arXiv},
      primaryClass={math.AG},
      url={https://arxiv.org/abs/2409.11153}, 
}

@article{bayeretal,
author = {Bayer, V. A. dos S. and de Guzmán, E. M. N. and Hefez, A. and Hernandes, M. E.},
title = {Tjurina number of a local complete intersection curve},
journal = {Communications in Algebra},
volume = {0},
number = {0},
pages = {1--12},
year = {2024},
publisher = {Taylor \& Francis},
doi = {10.1080/00927872.2024.2381823},
URL = {https://doi.org/10.1080/00927872.2024.2381823},
eprint = {https://doi.org/10.1080/00927872.2024.2381823}
}

@article {SaitoLogari,
    AUTHOR = {Saito, K.},
     TITLE = {Theory of logarithmic differential forms and logarithmic
              vector fields},
   JOURNAL = {J. Fac. Sci. Univ. Tokyo Sect. IA Math.},
  FJOURNAL = {Journal of the Faculty of Science. University of Tokyo.
              Section IA. Mathematics},
    VOLUME = {27},
      YEAR = {1980},
    NUMBER = {2},
     PAGES = {265--291},
      ISSN = {0040-8980},
   MRCLASS = {32G11 (14D05 32B30)},
  MRNUMBER = {586450},
MRREVIEWER = {Zoghman\ Mebkhout},
}

@book {Abexpansion,
    AUTHOR = {Abhyankar, S. S.},
     TITLE = {Lectures on expansion techniques in algebraic geometry},
    SERIES = {Tata Institute of Fundamental Research Lectures on Mathematics
              and Physics},
    VOLUME = {57},
      NOTE = {Notes by Balwant Singh},
 PUBLISHER = {Tata Institute of Fundamental Research, Bombay},
      YEAR = {1977},
     PAGES = {iv+168},
   MRCLASS = {14H20},
  MRNUMBER = {542446},
MRREVIEWER = {Jos\'e\ L.\ Vicente},
}

@misc {singular,
 title = {{\sc Singular} {4-3-0} --- {A} computer algebra system for polynomial computations},
 author = {Decker, W. and Greuel, G.-M. and Pfister, G. and Sch\"onemann, H.},
 year = {2022},
 howpublished = {\url{http://www.singular.uni-kl.de}},
}

@article {quotpat,
	AUTHOR = {Almir\'on, P.},
	TITLE = {On the quotient of {M}ilnor and {T}jurina numbers for
	two-dimensional isolated hypersurface singularities},
	JOURNAL = {Math. Nachr.},
	FJOURNAL = {Mathematische Nachrichten},
	VOLUME = {295},
	YEAR = {2022},
	NUMBER = {7},
	PAGES = {1254--1263},
	ISSN = {0025-584X,1522-2616},
	MRCLASS = {32S25 (14B05 32S05 32S50)},
	MRNUMBER = {4468864},
	MRREVIEWER = {Dmitry\ Kerner},
	DOI = {10.1002/mana.202100371},
	URL = {https://doi.org/10.1002/mana.202100371},
}

@article{LuengoPfister88,
    AUTHOR = {Luengo, I. and Pfister, G.},
     TITLE = {Normal forms and moduli spaces of curve singularities with
              semigroup {$\langle 2p,2q,2pq+d\rangle$}},
      NOTE = {Algebraic geometry (Berlin, 1988)},
   JOURNAL = {Compositio Math.},
  FJOURNAL = {Compositio Mathematica},
    VOLUME = {76},
      YEAR = {1990},
    NUMBER = {1-2},
     PAGES = {247--264},
      ISSN = {0010-437X,1570-5846},
   MRCLASS = {32S10 (14B05 14H20 32S05 32S30)},
  MRNUMBER = {1078865},
       URL = {http://www.numdam.org/item?id=CM_1990__76_1-2_247_0},
}

@article {heroy89,
    AUTHOR = {Her\"oy, H. O.},
     TITLE = {Local moduli for plane curve singularities, the dimension of
              the {$\tau$}-constant stratum},
   JOURNAL = {Math. Scand.},
  FJOURNAL = {Mathematica Scandinavica},
    VOLUME = {65},
      YEAR = {1989},
    NUMBER = {1},
     PAGES = {33--40},
      ISSN = {0025-5521,1903-1807},
   MRCLASS = {14B07 (32S30 32S50)},
  MRNUMBER = {1051820},
MRREVIEWER = {Er\ Jian\ Xiao},
       DOI = {10.7146/math.scand.a-12262},
       URL = {https://doi.org/10.7146/math.scand.a-12262},
}

@article {Altaumin,
    AUTHOR = {Alberich-Carrami\~nana, M. and Almir\'on, P. and Blanco, G. and Melle-Hern\'andez, A.},
     TITLE = {The minimal {T}jurina number of irreducible germs of plane
              curve singularities},
   JOURNAL = {Indiana Univ. Math. J.},
  FJOURNAL = {Indiana University Mathematics Journal},
    VOLUME = {70},
      YEAR = {2021},
    NUMBER = {4},
     PAGES = {1211--1220},
      ISSN = {0022-2518,1943-5258},
   MRCLASS = {14H20 (32S50)},
  MRNUMBER = {4318472},
MRREVIEWER = {Eugenii\ Shustin},
       DOI = {10.1512/iumj.2021.70.8583},
       URL = {https://doi.org/10.1512/iumj.2021.70.8583},
}

@article {GHtaumin,
    AUTHOR = {Genzmer, Y. and Hernandes, M. E.},
     TITLE = {On the {S}aito basis and the {T}jurina number for plane
              branches},
   JOURNAL = {Trans. Amer. Math. Soc.},
  FJOURNAL = {Transactions of the American Mathematical Society},
    VOLUME = {373},
      YEAR = {2020},
    NUMBER = {5},
     PAGES = {3693--3707},
      ISSN = {0002-9947,1088-6850},
   MRCLASS = {14H50 (14B05 32S05)},
  MRNUMBER = {4082253},
MRREVIEWER = {P.\ Schenzel},
       DOI = {10.1090/tran/8019},
       URL = {https://doi.org/10.1090/tran/8019},
}

@article {Genzmer22,
    AUTHOR = {Genzmer, Y.},
     TITLE = {Number of moduli for a union of smooth curves in {$(\Bbb C^2,
              0)$}},
   JOURNAL = {J. Symbolic Comput.},
  FJOURNAL = {Journal of Symbolic Computation},
    VOLUME = {113},
      YEAR = {2022},
     PAGES = {148--170},
      ISSN = {0747-7171,1095-855X},
   MRCLASS = {32S05},
  MRNUMBER = {4403070},
MRREVIEWER = {Sonia\ P\'erez-D\'iaz},
       DOI = {10.1016/j.jsc.2022.03.002},
       URL = {https://doi.org/10.1016/j.jsc.2022.03.002},
}

@misc{Genzmer24,
      title={The Saito vector field of a germ of complex plane curve}, 
      author={Genzmer, Y.},
      year={2024},
      eprint={2403.06587},
      archivePrefix={arXiv},
      primaryClass={math.DS},
      url={https://arxiv.org/abs/2403.06587}, 
}

@article {GHI23,
    AUTHOR = {Garc\'ia Barroso, E. R. and Hernandes, M. E. and Hern\'andez Iglesias, M. F.},
     TITLE = {Weierstrass 1-forms and nondicritical generalized curve
              foliations},
   JOURNAL = {Internat. J. Math.},
  FJOURNAL = {International Journal of Mathematics},
    VOLUME = {34},
      YEAR = {2023},
    NUMBER = {6},
     PAGES = {Paper No. 2350028, 20},
      ISSN = {0129-167X,1793-6519},
   MRCLASS = {32S65 (14H20)},
  MRNUMBER = {4591937},
MRREVIEWER = {Eugenii\ Shustin},
       DOI = {10.1142/S0129167X23500283},
       URL = {https://doi.org/10.1142/S0129167X23500283},
}


\end{document}